\documentclass[10pt,a4paper]{amsart}
\usepackage{amsmath,amssymb,amsthm}
\usepackage{enumitem, scalerel}
\usepackage[utf8]{inputenc}
\usepackage[T1]{fontenc}
\usepackage[english]{babel}
\usepackage{libertine}
\usepackage[libertine,cmintegrals,cmbraces]{newtxmath}

\usepackage{microtype}
\usepackage{xcolor}
\definecolor{darkred}{RGB}{139,0,0}
\definecolor{darkblue}{RGB}{0,0,139}
\definecolor{darkgreen}{RGB}{0,100,0}
\usepackage{hyperref}
\hypersetup{
  colorlinks   = true, 
  urlcolor     = darkred, 
  linkcolor    = darkred, 
  citecolor   = darkgreen}
\usepackage{tikz}
\usepackage{tikz-cd}
\usepackage[capitalise]{cleveref}

\newtheorem{bigthm}{Theorem}
\newtheorem{bigcor}[bigthm]{Corollary}

\newtheorem{thm}{Theorem}[section]
\newtheorem*{nthm}{Theorem}
\newtheorem{lem}[thm]{Lemma}
\newtheorem{conj}[thm]{Conjecture}
\newtheorem{prop}[thm]{Proposition}
\newtheorem{cor}[thm]{Corollary}

\theoremstyle{definition}
\newtheorem{dfn}[thm]{Definition}

\theoremstyle{remark}
\newtheorem{ex}[thm]{Example}

\newtheorem{rem}[thm]{Remark}
\newtheorem*{nrem}{Remark}

\setenumerate[1]{leftmargin=*,labelindent=0pt,label=(\roman*),}
\setitemize[1]{label=\raisebox{0.25ex}{\tiny$\bullet$}}

\tikzset{
  symbol/.style={
    draw=none,
    every to/.append style={
      edge node={node [sloped, allow upside down, auto=false]{$#1$}}}
  }
}

\newcommand{\id}{\ensuremath{\mathrm{id}}}
\newcommand{\im}{\ensuremath{\mathrm{im}}}

\newcommand{\fr}{\ensuremath{\mathrm{fr}}}
\newcommand{\Fr}{\ensuremath{\mathrm{Fr}^{\scaleobj{0.8}{+}}}}

\newcommand{\free}{\ensuremath{\mathrm{free}}}

\newcommand{\Maps}{\ensuremath{\mathrm{Map}}}

\newcommand{\MSO}{\ensuremath{\mathbf{MSO}}}

\newcommand{\HZ}{\ensuremath{\mathbf{HZ}}}

\newcommand{\SO}{\ensuremath{\mathrm{SO}}}

\newcommand{\BO}{\ensuremath{\mathrm{BO}}}
\newcommand{\BSO}{\ensuremath{\mathrm{BSO}}}

\newcommand{\Diff}{\ensuremath{\mathrm{Diff}^{\scaleobj{0.8}{+}}}}
\newcommand{\Diffuo}{\ensuremath{\mathrm{Diff}}}
\newcommand{\BDiff}{\ensuremath{\mathrm{BDiff}^{\scaleobj{0.8}{+}}}}
\newcommand{\BDiffuo}{\ensuremath{\mathrm{BDiff}}}

\newcommand{\hAut}{\ensuremath{\mathrm{hAut}^{\scaleobj{0.8}{+}}}}

\newcommand{\interior}[1]{\ensuremath{\mathrm{int}(#1)}}
\newcommand{\ord}{\ensuremath{\mathrm{ord}}}

\newcommand{\catsingle}[1]{\ensuremath{\mathcal{#1}}}

\newcommand{\oH}{\ensuremath{\mathrm{H}}}
\newcommand{\oO}{\ensuremath{\mathrm{O}}}
\newcommand{\oB}{\ensuremath{\mathrm{B}}}
\newcommand{\oT}{\ensuremath{\mathrm{T}}}

\newcommand{\bfC}{\ensuremath{\mathbf{C}}}

\newcommand{\bfF}{\ensuremath{\mathbf{F}}}

\newcommand{\bfH}{\ensuremath{\mathbf{H}}}

\newcommand{\bfO}{\ensuremath{\mathbf{O}}}

\newcommand{\bfR}{\ensuremath{\mathbf{R}}}
\newcommand{\bfZ}{\ensuremath{\mathbf{Z}}}

\newcommand{\cH}{\ensuremath{\catsingle{H}}}

\newcommand{\map}[5]{\ensuremath{#1\colon\begin{array}{rcl} 
      #2 & \longrightarrow & #3 \\[0.3em] 
      #4 & \longmapsto & #5
    \end{array}}}   

\newcommand{\mapnoname}[4]{\ensuremath{\begin{array}{rcl} 
      #1 & \longrightarrow & #2 \\[0.3em] 
      #3 & \longmapsto & #4
    \end{array}}}

\newcommand{\ra}{\rightarrow}
\newcommand{\lra}{\longrightarrow}

\newcommand{\xra}[1]{\xrightarrow{#1}}
\newcommand{\xlra}[1]{\overset{#1}{\longrightarrow}}

\newcommand{\longtwoheadrightarrow}{\relbar\joinrel\twoheadrightarrow}

\newcommand{\Mod}[1]{\ (\mathrm{mod}\ #1)}

\newcommand{\BSp}{\mathrm{BSp}}

\newcommand{\Sp}{\mathrm{Sp}}
\newcommand{\PSp}{\mathrm{PSp}}
\newcommand{\PSL}{\mathrm{PSL}}
\newcommand{\GL}{\mathrm{GL}}
\newcommand{\SL}{\mathrm{SL}}
\newcommand{\Aut}{\mathrm{Aut}}
\newcommand{\Hom}{\mathrm{Hom}}

\newcommand{\coker}{\mathrm{coker}}
\newcommand{\bP}{\mathrm{bP}}
\newcommand{\bA}{\mathrm{bA}}
\newcommand{\Ext}{\mathrm{Ext}}
\newcommand{\Tor}{\mathrm{Tor}}

\newcommand{\sgn}{\mathrm{sgn}}

\begin{document}

\title{Mapping class groups of highly connected $(4k+2)$-manifolds}
\begin{abstract}
\ We compute the mapping class group of the manifolds $\sharp^g(S^{2k+1}\times S^{2k+1})$ for $k>0$ in terms of the automorphism group of the middle homology and the group of homotopy $(4k+3)$-spheres. We furthermore identify its Torelli subgroup, determine the abelianisations, and relate our results to the group of homotopy equivalences of these manifolds.
\end{abstract}
\author{Manuel Krannich}
\email{\href{mailto:krannich@dpmms.cam.ac.uk}{krannich@dpmms.cam.ac.uk}}
\address{Centre for Mathematical Sciences, Wilberforce Road, Cambridge CB3 0WB, UK}
\maketitle

\addtocontents{toc}{\protect\setcounter{tocdepth}{0}}

The classical mapping class group $\Gamma_g$ of a genus $g$ surface naturally generalises to all even dimensions $2n$ as the group of isotopy classes
\[\Gamma_g^n=\pi_0\Diff(W_g)\] of orientation-preserving diffeomorphisms of the $g$-fold connected sum $W_g=\sharp^g(S^n\times S^n)$. Its action on the middle cohomology $ H(g)\coloneq\oH^n(W_g;\bfZ)\cong\bfZ^{2g}$ provides a homomorphism $\Gamma^n_g\ra \GL_{2g}(\bfZ)$ whose image is the symplectic group $\Sp_{2g}(\bfZ)$ in the surface case $2n=2$, and a certain arithmetic subgroup $G_g\subset\Sp_{2g}(\bfZ)$ or $G_g\subset\oO_{g,g}(\bfZ)$ in general, the description of which we shall recall later. The kernel $\oT_g^n\subset\Gamma_g^n$ of the resulting extension \begin{equation}\label{ses1}0 \lra \oT^n_g \lra \Gamma^n_g \lra G_g \lra 0\end{equation} is known as the \emph{Torelli group}---the subgroup of isotopy classes acting trivially on homology. In contrast to the surface case, the Torelli group in high dimensions $2n\ge6$ is comparatively manageable: there is an extension \begin{equation}\label{ses2}0 \lra \Theta_{2n+1} \lra \oT^n_g \lra  H(g)\otimes S\pi_n\SO(n) \lra 0\end{equation} due to Kreck \cite{Kreck}, which relates $\oT_g^n$ to the finite abelian group of homotopy spheres $\Theta_{2n+1}$ and the image of the stabilisation map $S\colon\pi_n\SO(n)\ra\pi_n\SO(n+1)$, shown in \cref{table:SO}.

The description of $\Gamma_g^n$ up to these two extension problems has found a variety of applications \cite{KauffmanKrylov, KrylovPseudo,BerglundMadsenI,BerglundMadsenII, EbertRW, operads,BotvinnikEbertWraith,Kupers,Grey,KrannichRW}, especially in relation to the study of moduli spaces of manifolds \cite{GRWusersguide}. The remaining extensions \eqref{ses1} and \eqref{ses2} have been studied more closely for particular values of $g$ and $n$ \cite{SatoDiff,Fried,Krylovthesis,Krylov,Crowley,GRWabelian} but are generally not well-understood (see e.g.\,\cite[p.1189]{Crowley},\cite[p.873]{GRWabelian},\cite[p.425]{operads}). In the present work, we resolve the remaining ambiguity for $n\ge3$ odd, resulting in a complete description of the mapping class group $\Gamma_g^n$ and the Torelli group $\oT_g^n$ in terms of the arithmetic group $G_g$ and the group of homotopy spheres $\Theta_{2n+1}$.

To explain our results, note that \eqref{ses1} and \eqref{ses2} induce further extensions
\begin{equation}\label{ses3}0 \lra \Theta_{2n+1} \lra \Gamma^n_g \lra \Gamma^n_g/\Theta_{2n+1} \lra 0\quad\text{and}\end{equation} 
\begin{equation}\label{ses4}0 \lra  H(g)\otimes S\pi_n\SO(n) \lra \Gamma^n_g/\Theta_{2n+1} \lra G_g \lra 0,\end{equation} which express $\Gamma_g^n$ in terms of $G_g$ and $\Theta_{2n+1}$ up to two extension problems that are similar to \eqref{ses1} and \eqref{ses2}, but are more convenient to analyse as both of their kernels are abelian. We resolve these two extension problems completely, beginning with an algebraic description of the second one in \cref{section:firstextension}, which enables us in particular to decide when it splits.

\begin{bigthm}\label{bigthm:splitextension}For $n\ge3$ odd and $g\ge1$, the extension
\[0 \lra  H(g)\otimes S\pi_n\SO(n) \lra \Gamma^n_g/\Theta_{2n+1} \lra G_g \lra 0\]
splits for $n\neq3,7$. For $n=3,7$, it splits if and only if $g=1$.
\end{bigthm}

Unlike \eqref{ses4}, the extension \eqref{ses3} is central and thus classified by a class in $\oH^2(\Gamma^n_g/\Theta_{2n+1};\Theta_{2n+1})$, which our main result, \cref{bigthm:mainextension} below, identifies in terms of three cohomology classes 
\begin{enumerate}
\item $\frac{\sgn}{8}\in\oH^2(\Gamma_g^n/\Theta_{2n+1};\bfZ)$ for $n\neq3,7$ odd,
\item $\frac{\chi^2}{2}\in\oH^2(\Gamma_g^n/\Theta_{2n+1};\bfZ)$ for $n\equiv3\Mod{4}$ and $n\neq3,7$,
\item $\frac{\chi^2-\sgn}{8}\in\oH^2(\Gamma_g^n/\Theta_{2n+1};\bfZ)$ for $n=3,7$,
\end{enumerate}
which can be expressed algebraically, using our description of $\Gamma^n_g/\Theta_{2n+1}$ mentioned above. Referring to \cref{section:mainsection} for the precise definition of these classes, we encourage the reader to think of them geometrically in terms of their pullbacks along the composition \[\BDiff(W_g)\lra\oB\Gamma_g^n\lra\oB(\Gamma_g^n/\Theta_{2n+1})\] induced by taking path components and quotients: the pullback of the first class, which is closely related to \emph{Meyer's signature cocycle} \cite{Meyer}, evaluates a class $[\pi]\in\oH_2(\BDiff(W_g);\bfZ)$ represented by an oriented smooth fibre bundle $\pi\colon E\ra S$ over a closed oriented surface $S$ with fibre $W_g$ to an eighth of the signature $\sgn(E)$ of its total space, the pullback of the second one assigns such a bundle the Pontryagin number $p_{(n+1)/4}^2(E)$ up to a fixed constant, and the pullback of the third class evaluates $[\pi]$ to a certain linear combination of $\sgn(E)$ and $p_{(n+1)/4}^2(E)$. In addition to these three classes, our identification of the extension \eqref{ses3} for $n\ge3$ odd involves two particular homotopy spheres: the first one, $\Sigma_P\in\Theta_{2n+1}$, is the \emph{Milnor sphere}---the boundary of the $E_8$-plumbing \cite[Sect.\,V]{BrowderBook}, and the second one, $\Sigma_Q\in\Theta_{2n+1}$, arises as the boundary of the plumbing of two copies of a linear $D^{n+1}$-bundles over $S^{n+1}$ classified by a generator of $S\pi_n\SO(n)$. We write $(-)\cdot\Sigma\colon\oH^2(\Gamma^n_g/\Theta_{2n+1};\bfZ)\ra\oH^2(\Gamma^n_g/\Theta_{2n+1};\Theta_{2n+1})$ for the change of coefficients induced by $\Sigma\in\Theta_{2n+1}$. 

\begin{table}[t]\label{table:SO}
  \centering
    \caption{The groups $S\pi_n\SO(n)$ for $n\ge3$, except that $S\pi_6\SO(6)=0$.} \label{table:SO}
  \begin{tabular}{ l | c c c c c c c c }
     $n\Mod{8}$ & $0$& $1$ & $2$ & $3$ & $4$ & $5$ &$6$   & $7$ \\  \hline
 $ S\pi_n\SO(n)$ &$(\bfZ/2)^2$ & $\bfZ/2$ &$\bfZ/2$ & $\bfZ$ & $\bfZ/2$ & $0$ &$\bfZ/2$ & $\bfZ$  \\
  \end{tabular}
\end{table}

\begin{bigthm}\label{bigthm:mainextension}For $n\ge3$ odd and $g\ge1$, the central extension
\[0 \lra \Theta_{2n+1} \lra \Gamma^n_g \lra \Gamma^n_g/\Theta_{2n+1} \lra 0\]
 is classified by
\begin{enumerate} 
\item $\frac{\sgn}{8}\cdot \Sigma_P \in \oH^2(\Gamma_g^n/\Theta_{2n+1};\Theta_{2n+1})$ for $n\equiv 1\Mod{4}$, 
\item $\frac{\sgn}{8}\cdot \Sigma_P + \frac{\chi^2}{2} \cdot \Sigma_Q \in \oH^2(\Gamma_g^n/\Theta_{2n+1};\Theta_{2n+1})$ for $n\equiv 3\Mod{4}$ if $n\neq 3,7$,
\item $\frac{\chi^2-\sgn}{8}\cdot \Sigma_Q \in \oH^2(\Gamma_g^n/\Theta_{2n+1};\Theta_{2n+1})$ for $n=3,7$.
\end{enumerate}
Moreover, this extension splits if and only if $n\equiv 1\Mod{4}$ and $g=1$.
\end{bigthm}

The extension \eqref{ses2} describing the Torelli group $\oT^n_{g}$ is the pullback of the extension determined in \cref{bigthm:mainextension} along the map $H(g)\otimes S\pi_n\SO(n)\rightarrow \Gamma_g^n/\Theta_{2n+1}$, so the combination of the previous result with our identification of $\Gamma_{g}^n/\Theta_{2n+1}$ provides an algebraic description of both $\Gamma_g^n$ and $\oT_g^n$ in terms of $G_g$ and $\Theta_{2n+1}$. We derive several consequences from this, beginning with deciding when the more commonly considered extensions \eqref{ses1} and \eqref{ses2} split.
\begin{bigcor}\label{bigthm:kreckextensions}Let $n\ge3$ odd and $g\ge1$.
\begin{enumerate} 
\item The extension
\[0 \lra \oT^n_g \lra \Gamma^n_g \lra G_g \lra 0\]
does not split for $g\ge2$, but admits a splitting for $g=1$ and $n\neq3,7$.
\item The extension 
\[0 \lra \Theta_{2n+1} \lra \oT^n_g \lra  H(g)\otimes S\pi_n\SO(n) \lra 0\] 
does not split for $n\equiv3\Mod{4}$, but splits $G_g$-equivariantly for $n\equiv 1 \Mod {4}$.
\end{enumerate}
\end{bigcor}

\subsection*{Abelian quotients}The second application of our description of the groups $\Gamma^n_{g}$ and $\oT^n_g$ is a computation of their abelianisations.

Although \cref{bigthm:splitextension} exhibits the extension \eqref{ses4} as being nontrivial in some cases, its abelianisation turns out to split nevertheless (see \cref{corollary:abelianisationhalfmcg}), so there exists a splitting
\[\oH_1(\Gamma_g^n/\Theta_{2n+1})\cong\oH_1(G_g)\oplus( H(g)\otimes S\pi_n\SO(n))_{G_g},\] which participates in the following identification of the abelianisation of $\Gamma_{g}^n$ and $\oT_{g}^n$.

\begin{bigcor}\label{bigthm:abstractsplitting}Let $g\ge1$ and $n\ge3$ odd.
\begin{enumerate}
\item The extension \eqref{ses3} induces a split short exact sequence
\[0\lra\Theta_{2n+1}/K_g\lra\oH_1(\Gamma_{g}^n)\xlra{p_*}\oH_1(G_g) \oplus \big(H(g)\otimes S\pi_n\SO(n)\big)_{G_g}\lra0,\] where $K_g=\langle \Sigma_P,\Sigma_Q\rangle$ for $g\ge2$ and $K_g=\langle \Sigma_Q\rangle$ for $g=1$.
\item The extension \eqref{ses4} induces a split short exact sequence of $G_g$-modules
\[0\lra\Theta_{2n+1}/\langle\Sigma_Q\rangle \lra \oH_1(\oT_{g}^n)\xlra{\rho_*} H(g)\otimes S\pi_n\SO(n)\lra0.\] In particular, the commutator subgroup of $\oT_{g}^n$ is generated by $\Sigma_Q$.
\end{enumerate}
\end{bigcor}

These splittings of $\oH_1(\Gamma_g^n)$ and $\oH_1(\oT_g^n)$ are constructed abstractly, but can often be made more concrete by means of a refinement of the mapping torus construction as a map \[t\colon \Gamma_g^n\lra\Omega^{\tau_{>n}}_{2n+1}\] to the bordism group of closed $(2n+1)$-manifolds $M$ equipped with a lift of their stable normal bundle $M\ra\BO$ to the $n$-connected cover $\tau_{>n}\BO\ra\BO$. To state the resulting more explicit description of the abelianisations of $\Gamma_g^n$ and $\oT_g^n$, we write $\sigma_n'$ for the minimal positive signature of a closed smooth $n$-connected $(2n+2)$-dimensional manifold. For $n\neq1,3,7$ odd, the intersection form of such a manifold is even, so $\sigma_n'$ is divisible by $8$.

\begin{bigcor}\label{bigthm:abeliansationmcg}Let $n\ge3$ odd and $g\ge1$.
\begin{enumerate}
\item The morphism
\[ t_*\oplus p_*\colon \oH_1(\Gamma^n_g)\lra \Omega^{\tau_{>n}}_{2n+1}\oplus \oH_1(G_g) \oplus \big(H(g)\otimes S\pi_n\SO(n)\big)_{G_g}\] is an isomorphism for $g\ge2$, and for $g=1$ if $n=3,7$. For $g=1$ and $n\neq3,7$, it is surjective, has kernel of order $\sigma_n'/8$ generated by $\Sigma_P$, and splits for $n\equiv 1\Mod{4}$.
\item The morphism
\[t_*\oplus \rho_*\colon \oH_1(\oT_g^n)\lra \Omega^{\tau_{>n}}_{2n+1}\oplus \big( H(g)\otimes S\pi_n\SO(n)\big)\] is an isomorphism for $n=3,7$. For $n\neq3,7$, it is surjective, has kernel of order $\sigma_n'/8$ generated by $\Sigma_P$, and splits $G_g$-equivariantly for $n\equiv 1\Mod{4}$.
\end{enumerate}
\end{bigcor}

\begin{nrem}\ \phantom{xyz}
\begin{enumerate}
\item In \cref{bigthm:abelianisationeven} below, we determine the abelianisation of $\Gamma^n_{g,1}$ and $\oT^n_{g,1}$ for $n\ge4$ even in which case the morphisms $t_*\oplus p_*$ and $t_*\oplus \rho_*$ are isomorphisms \emph{for all $g\ge1$} .
\item As shown in \cite[Prop.\,2.15]{KrannichReinhold}, the minimal signature $\sigma_n'$ is nontrivial for $n$ odd, grows at least exponentially with $n$, and can be expressed in terms of Bernoulli numbers.
\item For some values of $g$ and $n$, \cref{bigthm:abeliansationmcg} leaves open whether $t_*\oplus p_*$ and $t_*\oplus\rho_*$ split. The morphisms $p_*$ and $\rho_*$ always split by \cref{bigthm:abstractsplitting}, and in \cref{section:geometricsplitting} we relate the question of whether there exist compatible splittings of $t_*$ to a known open problem in the theory of highly connected manifold, showing in particular that such splittings do exist when assuming a conjecture of Galatius--Randal-Williams.
\item Work of Thurston \cite{Thurston} shows that the component $\Diffuo_0(W_g)\subset \Diff(W_g)$ of the identity is perfect as a discrete group, so the abelianisation of the full diffeomorphism group $\Diff(W_g)$ considered as a discrete group agrees with $\oH_1(\Gamma_g^n)$.
\end{enumerate}
\end{nrem}

In view of \cref{bigthm:abeliansationmcg}, it is of interest to determine the bordism groups $\Omega^{\tau_{>n}}_{2n+1}$, the abelianisation $\oH_1(G_g)$, and the coinvariants $(H(g)\otimes S\pi_n\SO(n))_{G_g}$. The computation
\begin{equation}\label{equ:coinvariants}
(H(g)\otimes S\pi_n\SO(n))_{G_g}\cong
\begin{cases}
0&g\ge2\text{ or }n=3,6,7\text{ or }n\equiv5\Mod{8}\\
\bfZ/2^2&g=1\text{ and }n\equiv0\Mod{8}\\
\bfZ/2&\text{otherwise}
\end{cases}
\end{equation} is straightforward (see \cref{lemma:coinvariantssymplectic} and \cref{table:SO}). The abelianisation of $G_g$ is  known and summarised in \cref{table:Aut} (see \cref{lemma:abelianisationSp}).
\begin{table}[t]
  \centering
    \caption{The abelianisation of $G_g$ for $n$ odd.} \label{table:Aut}
  \begin{tabular}{ l | c  c  c}
$\oH_1(G_g)$ & $g=1$ & $g=2$ & $g\ge3$\\\hline
$n=1,3,7$ & $\bfZ/12$ & $\bfZ/2$ & $0$\\
$n\neq 1,3,7$ odd & $\bfZ/4\oplus \bfZ$ & $\bfZ/4\oplus\bfZ/2$ & $\bfZ/4$
 \end{tabular}
\end{table}
Finally, the bordism groups $\Omega^{\tau_{>n}}_{2n+1}$ are closely connected to the stable homotopy groups of spheres: the canonical map \[\pi^s_{2n+1}\cong\Omega^{\fr}_{2n+1}\lra\Omega^{\tau_{>n}}_{2n+1}\] factors through the cokernel of the $J$-homomorphism and work of Schultz and Wall \cite{Schultz,Wall2n} implies that the induced morphism is often an isomorphism (see \cref{corollary:SchultzWall}).

\begin{nthm}[Schultz, Wall]\label{them:schultzwall}For $n$ odd, the natural morphism $\coker(J)_{2n+1}\rightarrow \Omega^{\tau_{>n}}_{2n+1}$ is surjective with cyclic kernel. It is an isomorphism for $n\equiv1\Mod{4}$ and for $n=3,7$.
\end{nthm}

Combined with \cref{bigthm:abeliansationmcg}, this reduces the computation of the abelianisation of $\Gamma_g^n$ and $\oT^n_g$ in many cases to determining the cokernel of the $J$-homomorphism---a well-studied problem in stable homotopy theory. \cref{table:exampleabelianisationstorelli} shows the resulting calculation of the abelianisations of the groups $\Gamma_g^n$ and $\oT^n_g$ for the first few values of $n$ .

\begin{table}[h]
  \centering
    \caption{Some abelianisations of $\oT_g^n$ and $\Gamma_g^n$.}\label{table:exampleabelianisations} \label{table:exampleabelianisationstorelli}
    \begin{tabular}{ l | c c c c }
  $\oH_1(\oT_g^n)$ & $n=3$ & $n=5$  & $n=7$ & $n=9$ \\  \hline
      $g=0$ & $\bfZ/28$ & $\bfZ/992$ &  $\bfZ/2\oplus \bfZ/8128 $ & $\bfZ/2\oplus\bfZ/261632$  \\
    $g\ge1$ & $\bfZ^{2g}$ & $\bfZ/992$ &  $\bfZ/2\oplus \bfZ^{2g}$ & $\bfZ/2\oplus(\bfZ/2)^{2g}\oplus\bfZ/261632$  \\
\end{tabular}

\bigskip

\medskip

\begin{tabular}{ l | c c c c }
  $\oH_1(\Gamma_g^n)$ & $n=3$ & $n=5$  & $n=7$ & $n=9$ \\  \hline
      $g=0$ & $\bfZ/28$ & $\bfZ/992$ &  $\bfZ/2\oplus \bfZ/8128 $ & $\bfZ/2\oplus\bfZ/261632$  \\
  $g=1$ & $\bfZ/12$ & $\bfZ/4\oplus\bfZ\oplus\bfZ/992$ &  $\bfZ/12\oplus \bfZ/2$ & $\bfZ/4\oplus\bfZ\oplus (\bfZ/2)^2\oplus \bfZ/261632$  \\
   $g=2$ & $\bfZ/2$ & $\bfZ/4\oplus\bfZ/2$ &  $(\bfZ/2)^2$ & $(\bfZ/2)^2\oplus\bfZ/4 $  \\
  $g\ge3$ & $0$ & $\bfZ/4$ &  $\bfZ/2$ & $\bfZ/2\oplus\bfZ/4 $  \\
\end{tabular}
\end{table}

\begin{nrem}After the completion of this work, Burklund--Hahn--Senger \cite{BHS} and Burklund--Senger \cite{BS} showed that for $n$ odd, the homotopy sphere $\Sigma_Q\in\Theta_{2n+1}$ bounds a parallelisable manifold if and only if $n\neq 11$. This implies in particular that aside from $n=11$
\begin{enumerate}
\item the canonical map $\coker(J)_{2n+1}\ra\Omega_{2n+1}^{\tau_{>n}}$ is an isomorphism, which extends the theorem attributed to Schultz and Wall above,
\item the conjecture of Galatius--Randal-Williams mentioned in the third part of the previous remark holds, and 
\item the minimal signature $\sigma_n'$ appearing in \cref{bigthm:abeliansationmcg} is computable from \cite[Prop.\,2.15]{KrannichReinhold}.
\end{enumerate} 
\end{nrem}

\subsection*{Homotopy equivalences}
As an additional application of our results, we briefly discuss the group $\pi_0\hAut(W_g)$ of homotopy classes of orientation-preserving homotopy equivalences. 

The natural map $\Gamma_g^n\ra\pi_0\hAut(W_g)$ can be seen to factor through the quotient $\Gamma_g^n/\Theta_{2n+1}$ and to induce a commutative diagram of the form \begin{equation}\label{ses5}
\begin{tikzcd}
0\arrow[r]& H(g)\otimes S\pi_{n}\SO(n)\arrow[r]\arrow[d]&\Gamma_g^n/\Theta_{2n+1}\arrow[r]\arrow[d]&G_g\arrow[r]\arrow[d,equal]& 0\\
0\arrow[r]& H(g)\otimes S\pi_{2n}S^{n}\arrow[r]&\pi_0\hAut(W_g)\arrow[r]&G_g\arrow[r]& 0,
\end{tikzcd}
\end{equation}
which exhibits the lower row---an extension describing $\pi_0\hAut(W_g)$ due to Baues \cite{Baues}---as the extension pushout of the extension \eqref{ses4} along the left vertical morphism, which is induced by the restriction $J\colon S\pi_n\SO(n)\ra S\pi_{2n}S^{n}$ of the unstable $J$-homomorphism, where $S\pi_{2n}S^{n}$ is the image of the suspension map $S:\pi_{2n}S^{n}\ra\pi_{2n+1}S^{n+1}$. By \cref{bigthm:splitextension}, the upper row splits in most cases and thus induces a compatible splitting of Baues' extension. In the cases in which the upper row does not split, we show that Baues' extension cannot split either.

\begin{bigcor}\label{bigcor:hAut}Let $n\ge3$ odd and $g\ge1$.\begin{enumerate}
\item For $n\neq3,7$, the two extensions in \eqref{ses5} admit compatible splittings. For $n=3,7$, either of the extensions splits if and only if $g=1$.
\item The induced morphism $\oH_1(\pi_0\hAut(W_g))\ra\oH_1(G_g)$ is an isomorphism for $g\ge2$, and a split epimorphism with kernel the coinvariants $(H(g)\otimes S\pi_{2n}S^{n})_{G_g}$, which vanish for $g\ge2$ or $n=3,7$, and agree with $S\pi_{2n}S^{n}/(2\cdot S\pi_{2n}S^{n})$ otherwise.
\end{enumerate}
\end{bigcor}

\subsection*{The groups $\Gamma_{g,1}^n$ for $n$ even}Some parts in our analysis of $\Gamma_{g,1}^n$ go through when $n\ge4$ is even as well, but a few key steps do not and would require new arguments. For instance, a different approach to the extension problem \eqref{ses4} would be necessary, as well as an extension of \cref{thm:almostclosedbundles} to incorporate the Arf invariant. The abelianisation of the groups $\Gamma^n_{g,1}$ and $\oT^n_{g,1}$, however, can be determined without fully solving the extensions \eqref{ses3} and \eqref{ses4} if $n$ is even. It turns out that in this case, the morphisms considered in \cref{bigthm:abeliansationmcg} are isomorphisms for \emph{all} $g\ge1$, which we shall prove as part of \cref{section:abelianisationneven}.

\begin{bigthm}\label{bigthm:abelianisationeven}For $n\ge4$ even and $g\ge1$, the morphisms 
\begin{gather*}t_*\oplus p_*\colon \oH_1(\Gamma_{g,1}^n) \lra \Omega^{\tau_{>n}}_{2n+1}\oplus \oH_1(G_g) \oplus \big( H(g)\otimes S\pi_n\SO(n)\big)_{G_g}\\
 t_*\oplus \rho_*\colon \oH_1(\oT_{g,1}^n) \lra \Omega^{\tau_{>n}}_{2n+1}\oplus\big(H(g)\otimes S\pi_n\SO(n)\big)\end{gather*}
 are isomorphisms for $g\ge1$. The coinvariants $\big( H(g)\otimes S\pi_n\SO(n)\big)_{G_g}$ are described in \eqref{equ:coinvariants}.
\end{bigthm}

\subsection*{Other highly connected manifolds}
Instead of restricting to $W_g$, one could consider any $(n-1)$-connected almost parallelisable manifold $M$ of dimension $2n\ge6$. Baues' and Kreck's work \cite{Baues,Kreck} applies in this generality, so there are analogues of the sequences \eqref{ses1}--\eqref{ses5} describing $\pi_0\Diff(M)$ and $\pi_0\hAut(M)$. However, for $n$ odd---the case of our main interest---Wall's classification of highly connected manifolds \cite{Wall2n} implies that any such manifold is diffeomorphic to a connected sum $W_g\sharp\Sigma$ with an exotic sphere $\Sigma\in\Theta_{2n}$, aside from those of Kervaire invariant $1$, which only exist in dimensions $6$, $14$, $30$, $62$, and possibly $126$ by work of Hill--Hopkins--Ravenel \cite{HHR}. The mapping class group $\pi_0\Diff(W_g\sharp\Sigma)$ for $\Sigma\in\Theta_{2n}$ and $n$ odd in turn is completely understood in terms of $\Gamma_g^n$: Kreck's work \cite[Lem.\,3, Thm 3]{Kreck} shows that the former is a quotient of the latter by a known element $\Sigma'\in\Theta_{2n+1}$ of order at most $2$, which is trivial if and only if $\eta\cdot[\Sigma]\in\coker(J)_{2n+1}$ vanishes.

\subsection*{Previous results}The extensions \eqref{ses1} and \eqref{ses2} and their variants \eqref{ses3} and \eqref{ses4} have been studied by various authors before, and some special cases of our results were already known:
\begin{enumerate}
\item As an application of their programme on moduli spaces of manifolds, Galatius--Randal-Williams \cite{GRWabelian} determined the abelianisation of $\Gamma_g^n$ for $g\ge5$ and used this to determine the extension \eqref{ses3} for $n\equiv 5\Mod{8}$ up to automorphisms of $\Theta_{2n+1}$ as long $g\ge5$. Our work recovers and extends their results, also applies to low genera $g<5$, and does not rely on their work on moduli spaces of manifolds.
\item Theorems~\ref{bigthm:splitextension} and~\ref{bigcor:hAut} for $n=3,7$ reprove results due to Crowley \cite{Crowley}. 
\item Baues \cite[Thm\,8.14, Thm\,10.3]{Baues} showed that the lower extension in \eqref{ses5} splits for $n\neq3,7$ odd, which we recover as part of the first part of \cref{bigcor:hAut}.
\item The case $(g,n)=(1,3)$ of \cref{bigthm:splitextension} and \cref{bigthm:kreckextensions} (ii) can be deduced from work of Krylov \cite{Krylov} and Fried \cite{Fried}, who also showed that the extension of \cref{bigthm:kreckextensions} (i) does not split in this case. Krylov \cite[Thms\,2.1, 3.2, 3.3]{Krylovthesis} moreover established the case $n\equiv 5\Mod{8}$ of \cref{bigthm:kreckextensions} (i) for $g=1$. For $n\neq3,7$, he also proved the case $n\equiv3\Mod{4}$ of \cref{bigthm:splitextension} and the case $n\equiv3\Mod{4}$ of \cref{bigthm:kreckextensions} (ii) for $g=1$.
\end{enumerate}

\subsection*{Further applications}Our main result \cref{bigthm:mainextension} has been used in \cite{KrannichKupers} in conjunction with Galatius--Randal-Williams' work on moduli spaces of manifolds \cite{GRWusersguide} to compute the second stable homology of the \emph{theta-subgroup} of $\Sp_{2g}(\bfZ)$ (see \cref{section:Wallsform}), or equivalently, the second quadratic symplectic algebraic $K$-theory group of the integers $\mathrm{KSp}^q_2(\mathbf{Z})$.

\subsection*{Outline}\cref{section:foundations} serves to recall foundational material on diffeomorphism groups and their classifying spaces, as well as to introduce different variants of the extensions \eqref{ses1} and \eqref{ses2} and to establish some of their basic properties. In \cref{section:firstextension}, we study the action of $\Gamma_g^n$ on the set of stable framings of $W_g$ to identify the extension \eqref{ses3} and prove \cref{bigthm:splitextension}. \cref{section:mainsection} aims at the proof of our main result \cref{bigthm:mainextension}, which requires some preparation. We recall the relation between relative Pontryagin classes and obstruction theory in \cref{section:obstructionclass}, discuss aspects of Wall's classification of highly connected manifolds in \cref{section:highlyconnected}, relate this class of manifolds to $W_{g,1}$-bundles over surfaces with certain boundary conditions in \cref{section:extensionclass} (which incidentally is the key geometric insight to prove \cref{bigthm:mainextension}), construct the cohomology classes appearing in the statement of \cref{bigthm:mainextension} in Sections~\ref{section:algebraicsignature} and~\ref{section:algebraicobstructions}, and finish with the proof of \cref{bigthm:mainextension} in \cref{section:proofmainthm}. In \cref{section:applications}, we analyse the extensions \eqref{ses1} and \eqref{ses2} and compute the abelianisation of $\Gamma^n_{g}$ and $\oT_g^n$, proving Corollaries~\ref{bigthm:kreckextensions}--\ref{bigthm:abeliansationmcg} and \cref{bigthm:abelianisationeven}. \cref{section:homotopyaut} briefly discusses the group of homotopy equivalences and proves \cref{bigcor:hAut}. In the appendix, we compute various low-degree (co)homology groups of the symplectic group $\Sp_{2g}(\bfZ)$ and its arithmetic subgroup $G_g\subset \Sp_{2g}(\bfZ)$.

\subsection*{Acknowledgements}I would like to thank Oscar Randal-Williams for several valuable discussions, Aurélien Djament for an explanation of an application of a result of his, and Fabian Hebestreit for many useful comments on an earlier version of this work. I was supported by the European Research Council (ERC) under the European Union’s Horizon 2020 research and innovation programme (grant agreement No 756444).
\numberwithin{equation}{section}

\addtocontents{toc}{\protect\setcounter{tocdepth}{1}}

\bigskip

\tableofcontents

\clearpage

\section{Variations on two extensions of Kreck}\label{section:foundations}
\subsection{Different flavours of diffeomorphisms}\label{section:differentdiffeos}
Throughout this work, we write
\[W_g=\sharp^{g}(S^n\times S^n)\quad\text{and}\quad W_{g,1}=\sharp^{g}(S^n\times S^n)\backslash \interior{D^{2n}}\] for the $g$-fold connected sum of $S^n\times S^n$, including $W_0=S^{2n}$, and the manifold obtained from $W_g$ removing the interior of an embedded disc $D^{2n}\subset W_{g}$. Occasionally, we view the manifold $W_{g,1}$ alternatively as the iterated boundary connected sum $W_{g,1}=\natural^gW_{1,1}$ of $W_{1,1}=S^n\times S^n\backslash \interior{D^{2n}}$. We call $g$ the \emph{genus} of $W_g$ or $W_{g,1}$ and denote by $\Diff(W_g)$ and $\Diff(W_{g,1})$ the groups of orientation-preserving diffeomorphisms, not necessarily fixing the boundary in the case of $W_{g,1}$. We shall also consider the subgroups \[\Diffuo_\partial(W_{g,1})\subset\Diffuo_{\partial/2}(W_{g,1})\subset\Diff(W_{g,1})\] of diffeomorphisms required to fix a neighbourhood of the boundary $\partial W_{g,1}\cong S^{2n-1}$ or a neighbourhood of a chosen disc $D^{2n-1}\subset \partial W_{g,1}$ in the boundary, respectively. All groups of diffeomorphisms are implicitly equipped with the smooth topology so that 
\begin{enumerate}
\item $\BDiff(W_{g,1})$ and $\BDiff(W_{g})$ classify smooth oriented $W_{g,1}$-bundles or $W_g$-bundles,
\item $\BDiffuo_{\partial}(W_{g,1})$ classifies $(W_{g,1},S^{2n-1})$-bundles, i.e.\,smooth $W_{g,1}$-bundles with a trivialisation of their $S^{2n-1}$-bundle of boundaries, and 
\item $\BDiffuo_{\partial/2}(W_{g,1})$ classifies $(W_{g,1},D^{2n-1})$-bundles, that is, smooth $W_{g,1}$-bundles with a trivialised $D^{2n-1}$-subbundle of its $S^{2n-1}$-bundle of boundaries.
\end{enumerate}

Taking path components, we obtain various groups of isotopy classes
\[\Gamma_g^n=\pi_0\Diff(W_g),\quad\Gamma_{g,1/2}^n=\pi_0\Diffuo_{\partial/2}(W_g),\quad
\text{and}\quad\Gamma_{g,1}^n=\pi_0\Diffuo_\partial(W_{g,1}).\] Extending diffeomorphisms by the identity provides a map $\Diffuo_\partial(W_{g,1})\rightarrow \Diff(W_g)$, which induces an isomorphism on path components by work of Kreck as long as $n\ge3$.
\begin{lem}[Kreck]\label{lemma:fixingadiscdoesnotmatter}
The canonical map $\Gamma^n_{g,1}\rightarrow\Gamma^n_g$ is an isomorphism for $n\ge3$.
\end{lem}
\begin{proof}Taking the differential at the centre of the disc induces a fibration $\Diff(W_g)\ra \Fr(W_g)$ to the oriented frame bundle of $W_g$. Its fibre is the subgroup of diffeomorphisms that fix a point and its tangent space, so it is equivalent to the subgroup of diffeomorphisms fixing a small disc around that point, which is in turn equivalent to $\Diffuo_\partial(W_{g,1})$. We thus have fibration sequences of the form
\[\Diffuo_\partial(W_{g,1})\lra \Diff(W_{g,1})\lra \Fr(W_g)\quad\text{and}\quad\SO(2n)\lra\Fr(W_g)\lra W_g\] whose long exact sequences show that the morphism in question is surjective and also that its kernel is generated by a single isotopy class given by ``twisting'' a collar $[0,1]\times S^{2n-1}\subset W_{g,1}$ using a smooth based loop in $\SO(2n)$ that represents a generator of $\pi_1\SO(2n)\cong\bfZ/2$ (see \cite[p.\,647]{Kreck}). It follows from \cite[Lem.\,3 b), Lem.\,4]{Kreck} that this isotopy class is trivial since $W_g$ bounds the parallelisable handlebody $\natural^g(D^{n+1}\times S^n)$.
\end{proof}

For the purpose of studying the mapping class group $\Gamma_g^n$, we can thus equally well work with $\Diffuo_\partial(W_{g,1})$ instead of $\Diff(W_{g})$, which is advantageous since there is a stabilisation map $\Diffuo_{\partial/2}(W_{g,1})\rightarrow\Diffuo_{\partial/2}(W_{g+1,1})$ by extending diffeomorphisms over an additional boundary connected summand via the identity, which restricts to a map $\Diffuo_{\partial}(W_{g,1})\ra\Diffuo_{\partial}(W_{g+1,1})$ and thus induces stabilisation maps of the form \begin{equation}\label{equation:stabilisation}\Gamma^n_{g,1/2}\lra\Gamma^n_{g,1/2}\quad\text{ and }·\quad\Gamma^n_{g,1}\lra\Gamma^n_{g+1,1},\end{equation} that allow us to compare mapping class groups of different genera. The group $\Gamma_{0,1}$ has a convenient alternative description: gluing two closed $d$-discs along their boundaries via a diffeomorphism supported in a disc $D^{d}\subset \partial D^{d+1}$ gives a morphism $\pi_0\Diffuo_\partial(D^{d})\lra\Theta_{d+1}$ to Kervaire--Milnor's \cite{KervaireMilnor} finite abelian group $\Theta_d$ of oriented homotopy $d$-spheres up to $h$-cobordism. By work of Cerf \cite{Cerf}, this is an isomorphism for $d\ge5$, so we identify these two groups henceforth. Iterating the stabilisation map yields maps
\begin{equation}\label{equation:thetainclusion}\Theta_{2n+1}=\pi_0\Diffuo_{\partial}(D^{2n})=\Gamma^n_{0,1}\lra\Gamma_{g,1}^n\lra \Gamma_{g,1/2}^n.\end{equation} 

\begin{lem}\label{lemma:central}For $n\ge3$, the image of $\Theta_{2n+1}$ in $\Gamma_{g,1}^n$ is central and becomes trivial in $\Gamma_{g,1/2}^n$. The induced morphism \[\Big(\Gamma^n_{g,1}/(\im(\Theta_{2n+1}\ra\Gamma^n_{g,1})\Big)\lra\Gamma_{g,1/2}^n\] is an isomorphism.
\end{lem}

\begin{proof}Every diffeomorphism of $W_{g,1}$ supported in a disc is isotopic to one that is supported in an arbitrary small neighbourhood of the boundary and thus commutes with any diffeomorphism in $\Diffuo_\partial(W_{g,1})$ up to isotopy, which shows the first part of the claim. For the others, we consider the sequence of topological groups
\begin{equation}\label{equation:restrictionfibration}\Diffuo_\partial(W_{g,1})\lra\Diffuo_{\partial/2}(W_{g,1})\lra\Diffuo_{\partial}(D^{2n-1})\end{equation} induced by restricting diffeomorphisms in $\Diffuo_{\partial/2}(W_{g,1})$ to the moving part of the boundary. This is a fibration sequence by the parametrised isotopy extension theorem. Mapping this sequence for $g=0$ into \eqref{equation:restrictionfibration} via the iterated stabilisation map, we see that the looped map $\Omega\Diffuo_\partial(D^{2n-1})\ra\Diffuo_{\partial}(W_{g,1})$ induced by \eqref{equation:restrictionfibration} factors as the composition
\[\Omega\Diffuo_\partial(D^{2n-1})\lra\Diffuo_\partial(D^{2n})\lra\Diffuo_{\partial}(W_{g,1})\] of the map defining the Gromoll filtration with the iterated stabilisation map. Since the first map in this factorisation is surjective on path components by Cerf's work  \cite{Cerf}, the claim will follow from the long exact sequence on homotopy groups of \eqref{equation:restrictionfibration} once we show that the map $\Gamma_{g,1/2}^n\ra\pi_0\Diffuo_{\partial}(D^{2n-1})=\Theta_{2n}$ has trivial image. Using that any orientation preserving diffeomorphism fixes any chosen oriented codimension $0$ disc up isotopy by the isotopy extension theorem, one easily sees that this image agrees with the the inertia group of $W_g$, which is known to vanish by work of Kosinski and Wall \cite{Kosinski,WallInertia}.
\end{proof}

\subsection{Wall's quadratic form}\label{section:Wallsform}
We recall Wall's quadratic from associated to an $(n-1)$-connected $2n$-manifold \cite{Wall2n}, specialised to the case of our interest---the iterated connected sums $W_g=\sharp^g(S^n\times S^n)$ in dimension $2n\ge6$.

The intersection form $\lambda\colon  H(g){\otimes}H(g)\rightarrow \bfZ$ on the middle cohomology $ H(g)\coloneq\oH^n(W_g;\bfZ)$ is a nondegenerate $(-1)^{n}$-symmetric bilinear form. We use Poincaré duality to identify $H(g)$ with $\pi_n(W_g;\bfZ)\cong \oH_n(W_g)$ and a result of Haefliger \cite{Haefliger} to represent classes in $\pi_n(W_g)$  by embedded spheres $e\colon S^n \hookrightarrow W_g$, unique up to isotopy as long as $n\ge4$. As $W_g$ is stably parallelisable, the normal bundle of such $e$ is stably trivial and hence gives a class \[q([e])\in\ker\big(\pi_n(\BSO(n))\rightarrow\pi_n(\BSO(n+1)\big)\cong\bfZ/\Lambda_n=\begin{cases}
\bfZ&\mbox{if }n\text{ even}\\
\bfZ/2&\mbox{if }n\text{ odd},n\neq 3,7\\
0&\mbox{if }n=3,7.
\end{cases},\] where $\Lambda_n$ is the image of the usual map $\pi_n(\SO(n+1))\ra\pi_n(S^n)\cong\bfZ$ (see e.g.\,\cite[§1.B]{Levine}).\footnote{Note that $\bfZ/\Lambda_n$ is the trivial in the case $n=3$ where Haefliger's result does not apply.} This defines a function $q\colon  H(g)\lra \bfZ/\Lambda_n$, which Wall \cite{Wall2n} showed to satisfy
\begin{enumerate}
\item $q(k\cdot[e])=k^2\cdot q([e])$ and
\item $q([e]+[f])=q([e])+q([f])+\lambda([e],[f]) \Mod{\Lambda_n}$.
\end{enumerate}
Note that for $n$ even, (i) and (ii) together imply $q([e])=\frac{1}{2}\lambda([e],[e])\in\bfZ$, so $q$ can in this case be recovered from $\lambda$. The triple $(H(g),\lambda, q)$ is the \emph{quadratic form} associated to $W_g$. The decomposition $W_g=\sharp^g (S^n\times S^n)$ into connected summands induces a basis $(e_1,\ldots,e_g,f_1,\ldots f_g)$ of $ H(g)\cong \bfZ^{2g}$ with respect to which $q$ and $\lambda$ have the form
\[\map{q}{\bfZ^{2g}}{\bfZ/\Lambda_n}{(x_1,\ldots x_g,y_1,\ldots,y_g)}{\sum_{i=1}^gx_iy_i}\quad\text{and}\quad J_{g,(-1)^n}=\left(\begin{matrix} 0 & I \\ (-1)^n I & 0 \end{matrix}\right),\] so the automorphism group of the quadratic form can be identified as \[G_g\coloneq\Aut\big(H(g),\lambda,q\big)\cong\begin{cases}
\oO_{g,g}(\bfZ)&n\text{ even}\\
\Sp_{2g}^q(\bfZ)&n\text{ odd},n\neq 3,7\\
\Sp_{2g}(\bfZ)&n=3,7,
\end{cases}\]  where \begin{align*}&\oO_{g,g}(\bfZ)=\{A\in\bfZ^{2g\times 2g}\mid A^T J_{g,1} A=J_{g,1}\},\\ &\Sp_{2g}(\bfZ)=\{A\in\bfZ^{2g\times 2g}\mid A^T J_{g,-1} A=J_{g,-1}\},\\
&\Sp_{2g}^q(\bfZ)=\{A\in\Sp_{2g}(\bfZ)\mid qA=q\}.\end{align*}
In the theory of theta-functions, the finite index subgroup $\Sp_{2g}^q(\bfZ)\le\Sp_{2g}(\bfZ)$ is known as the \emph{theta group}; it is the stabiliser of the standard theta-characteristic with respect to the transitive $\Sp_{2g}(\bfZ)$-action on the set of even characteristics (see e.g.\,\cite{Weintraub}). Using this description, it is straightforward to compute its index in $\Sp_{2g}(\bfZ)$ to be $2^{2g-1}+2^{g-1}$. 

\subsection{Kreck's extensions}\label{section:krecksextensions}To recall Kreck's extensions \cite[Prop.\,3]{Kreck} describing $\Gamma_{g,1}^n$ for $n\ge3$, note that an orientation-preserving diffeomorphism of $W_g$ induces an automorphism of the quadratic form $( H(g),\lambda,q)$. This provides a morphism $\Gamma_{g,1}^n\rightarrow G_g$, which Kreck proved to be surjective using work of Wall \cite{WallII}.\footnote{Kreck phrases his results in terms of pseudo-isotopy instead of isotopy. By Cerf's ``pseudo-isotopy implies isotopy'' \cite{Cerf}, this does not make a difference as long as $n\ge3$.} This explains the first extension\begin{equation}\label{equation:kreckses1}0\lra \oT_{g,1}^n\lra \Gamma_{g,1}^n\lra G_g\lra 0.\end{equation} The second extension describes the Torelli subgroup $\oT_{g,1}^n\subset \Gamma_{g,1}^n$ and has the form \begin{equation}\label{equation:kreckses2}0\lra \Theta_{2n+1}\lra \oT_{g,1}^n\xlra{\rho} H(g)\otimes S\pi_n\SO(n)\lra 0,\end{equation} where $S\pi_n\SO(n)$ denotes the image of the morphism $S\colon\pi_n\SO(n)\ra\pi_n\SO(n+1)$ induced by the usual inclusion $\SO(n)\subset\SO(n+1)$. The isomorphism type of this image can be extracted from work of Kervaire \cite{Kervaire} to be as shown in \cref{table:SO}. As a diffeomorphism supported in a disc acts trivially on cohomology, the morphism $\Theta_{2n+1}\ra\Gamma_{g,1}^n$ in \eqref{equation:thetainclusion} has image in $\oT_{g,1}^n$, which explains first map in the extension \eqref{equation:kreckses2}. To define the second one, we canonically identify $ H(g)\otimes S\pi_n\SO(n)$ with $\Hom(\oH_n(W_g;\bfZ),S\pi_n\SO(n))$ using that $\oH_n(W_g;\bfZ)$ is free and note that for a given isotopy class $[\phi]\in\oT_g^n$ and a class $[e]\in \oH_n(W_g;\bfZ)$ represented by an embedded sphere $e\colon S^n\hookrightarrow W_g$, the embedding $\phi\circ e$ is isotopic to $e$, so we can assume that $\phi$ fixes $e$ pointwise by the isotopy extension theorem. The derivative of $\phi$ thus induces an automorphism of the once stabilised normal bundle $\vartheta(e)\oplus \varepsilon$, which after choosing a trivialisation $\vartheta(e)\oplus \varepsilon\cong \varepsilon^{n+1}$ defines the image $\rho([\phi])([e])\in\pi_n\SO(n+1)$ of $[e]$ under the morphism $\rho([\phi])\in\Hom(\oH_n(W_g;\bfZ),S\pi_n\SO(n))$, noting that $\rho([\phi])([e])$ is independent of all choices and actually lies in the subgroup $S\pi_n\SO(n)\subset \pi_n\SO(n+1)$ (see \cite[Lem.\,1]{Kreck}).

Instead of the extensions \eqref{equation:kreckses1} and \eqref{equation:kreckses2}, we shall mostly be concerned with two closely related variants which we describe now. By Kreck's result, the morphism $\Theta_{2n+1}\ra\Gamma_{g,1}^n$ is injective, so gives rise to an extension $0\ra\Theta_{2n+1}\ra\Gamma^n_{g,1}\ra\Gamma^n_{g,1}/\Theta_{2n+1}\ra0$, which combined with the canonical identification $\Gamma^n_{g,1}/\Theta_{2n+1}\cong\Gamma^n_{g,1/2}$ of \cref{lemma:central} has the form
\begin{equation}\label{equation:mainextension2}0\lra\Theta_{2n+1}\lra\Gamma^n_{g,1}\lra\Gamma^n_{g,1/2}\lra0\end{equation} and agrees with the extension induced by taking path components of the chain of inclusions $\Diffuo_{\partial}(D^{2n})\subset\Diffuo_{\partial}(W_{g,1})\subset\Diffuo_{\partial/2}(W_{g,1})$. The action of $\Gamma^n_{g,1/2}$ on $H(g)$ preserves the quadratic form as $\Gamma^n_{g,1}\ra \Gamma^n_{g,1/2}$ is surjective by \cref{lemma:central}, so this action yields an extension \begin{equation}\label{equation:splitextension2}
0\lra H(g)\otimes S\pi_n\SO(n)\lra \Gamma^n_{g,1/2}\xlra{p} G_g\lra0,
\end{equation} which, via the isomorphism $\oT_{g,1}^n/\Theta_{2n+1}\cong H(g)\otimes S\pi_n\SO(n)$ induced by \eqref{equation:kreckses2}, corresponds to the quotient of the extension \eqref{equation:kreckses1} by $\Theta_{2n+1}$, using $\Gamma^n_{g,1}/\Theta_{2n+1}\cong\Gamma^n_{g,1/2}$ once more. 

\begin{lem}The action of $G_g$ on $H(g)\otimes S\pi_n\SO(n)\cong \Hom(\oH_n(W_{g};\bfZ),S\pi_n\SO(n))$ induced by the extension \eqref{equation:splitextension2} is through the standard action of $G_g$ on $\oH_n(W_{g};\bfZ)$.
\end{lem}
\begin{proof}
In view of the commutative diagram
\begin{center}
\begin{tikzcd}
0\arrow[r]&\oT_{g,1}^n\arrow[r]\arrow[d,"\rho"]& \Gamma_{g,1}^n\arrow[r]\arrow[d]& G_g\arrow[d,equal]\arrow[r]&0\\
0\arrow[r]&H(g)\otimes S\pi_n\SO(n)\arrow[r]& \Gamma^n_{g,1/2}\arrow[r,"p"]&G_g\arrow[r]&0,
\end{tikzcd}
\end{center}
it suffices to establish the identity $\rho(\phi \circ \psi\circ \phi^{-1})=p(\phi)\cdot \rho(\psi)$ for all $\phi\in \Gamma^n_{g,1}$ and $\psi\in \oT_{g,1}^n$. Unwrapping the definition of $\rho$, the image of $p(\phi)\cdot \rho(\psi)$ on a homology class in $\oH_n(W_{g,1};\bfZ)$ is given by the automorphism
\[\varepsilon^{n+1}\xlra{F^{-1}}\vartheta(\phi^{-1}\circ e)\oplus \varepsilon\xlra{d(\psi)}\vartheta(\phi^{-1}\circ e)\oplus \varepsilon\xlra{F}\varepsilon^{n+1},\]
where $e$ is an embedded sphere which represents the homology class and is pointwise fixed by $\phi\circ \psi\circ \phi^{-1}$ and $F$ is any choice of framing of $\vartheta(\phi^{-1}\circ e)\oplus \varepsilon$. Choosing the framing
\[\vartheta(e)\oplus \varepsilon\xlra{d(\phi^{-1})\oplus \varepsilon}\vartheta(\phi^{-1}\circ e)\oplus \varepsilon\xlra{F}\varepsilon^{n+1}\] 
to compute the image of $[e]\in\oH_n(W_{g,1};\bfZ)$ under $\rho(\phi \circ \psi\circ \phi^{-1})$, the claimed identity is a consequence of the chain rule for the differential.
\end{proof}

\subsection{Stabilisation}\label{section:stabilisation}Iterating the stabilisation map \eqref{equation:stabilisation} induces a morphism \begin{center}\begin{tikzcd}
0\arrow[r]&\Theta_{2n+1}\arrow[r]\arrow[d,equal]&\Gamma^n_{h,1}\arrow[r]\arrow[d]&\Gamma^n_{h,1/2}\arrow[d]\arrow[r]&0\\
0\arrow[r]&\Theta_{2n+1}\arrow[r]&\Gamma^n_{g,1}\arrow[r]&\Gamma^n_{g,1/2}\arrow[r]&0,
\end{tikzcd}\end{center} of group extensions for $h\le g$, which exhibits the upper row as the pullback of the lower row, so the extension \eqref{equation:mainextension2} for a fixed genus $g$ determines those for all $h\le g$. The situation for the extension \eqref{equation:splitextension2} is similar: writing $W_{g,1}\cong W_{h,1}\natural W_{g-h,1}$, we obtain a decomposition $H(g)\cong H(h)\oplus H(g-h)$, which yields a stabilisation map \[s\coloneq(-)\oplus \id_{H(g-h)}\colon G_h\lra G_{g}\] and two morphisms of $G_h$-modules, an inclusion $H(h)\ra s^*H(g)$ and a projection $s^*H(g)\ra H(h)$. These morphisms express the extension \eqref{equation:splitextension2} for genus $h\le g$ as being obtained from that for genus $g$ by pulling back along $s\colon G_h\ra G_{g}$ followed by forming the extension pushout along $s^*H(g)\ra H(h)$. They also induce a morphism of the form
\begin{center}
\begin{tikzcd}
0\arrow[r]&H(h)\otimes S\pi_n\SO(n)\arrow[r]\arrow[d]&\Gamma^n_{h,1/2}\arrow[r]\arrow[d]&G_h\arrow[d]\arrow[r]&0\\
0\arrow[r]&H(g)\otimes S\pi_n\SO(n)\arrow[r]&\Gamma^n_{g,1/2}\arrow[r]&G_{g}\arrow[r]&0.
\end{tikzcd}\end{center}

\section{The action on the set of stable framings and \cref{bigthm:splitextension}}
\label{section:firstextension}This section serves to resolve the extension problem
\begin{equation}\label{equation:Crowleyextension}0\lra H(g)\otimes S\pi_n\SO(n)\lra\Gamma_{g,1/2}^n\lra G_g\lra 0,\end{equation} described in the previous section. Our approach is in parts inspired by work of Crowley \cite{Crowley}, who identified this extension in the case  $n=3,7$.

The group $\Gamma^n_{g,1/2}$ acts on the set of equivalence classes of stable framings \[F\colon TW_{g,1}\oplus\varepsilon^k\cong\varepsilon^{2n+k}\quad \text{for }k\gg0\] that extend the standard stable framing on $TD^{2n-1}$, by pulling back stable framings along the derivative. As the equivalence classes of such framings naturally form a torsor for the group of pointed homotopy classes $[W_{g,1},\SO]_*$, the action of $\Gamma^n_{g,1/2}$ on a fixed choice of stable framing $F$ as above yields a function \[s_F\colon \pi_0\Diffuo_{\partial/2}(W_{g,1})\lra[W_{g,1},\SO]_*\cong\Hom(\oH_n(W_g;\bfZ),\pi_n\SO)\cong  H(g)\otimes\pi_n\SO,\] where the first isomorphism is induced by $\pi_n(-)$ and the Hurewicz isomorphism, and the second one by the universal coefficient theorem. This function is a $1$-cocycle (or \emph{crossed homomorphism}) for the canonical action of $\Gamma^n_{g,1/2}$ on $ H(g)\otimes \pi_n\SO$ (cf.\,\cite[Prop.\,3.1]{Crowley}) and as this action factors through the map $p\colon \Gamma^n_{g,1/2}\ra G_g$, we obtain a morphism of the form \[(s_F,p)\colon \Gamma^n_{g,1/2}\rightarrow ( H(g)\otimes \pi_n\SO)\rtimes G_g,\] which is independent of $F$ up to conjugation in the target by a straightforward check. This induces a morphism from \eqref{equation:Crowleyextension} to the trivial extension of $G_g$ by the $G_g$-module $H(g)\otimes \pi_n\SO$, \begin{equation}\label{equation:stableframing}
\begin{tikzcd}[column sep=0.4cm]
0\arrow[r]& H(g)\otimes S\pi_n\SO(n)\arrow[d]\arrow[r]&\Gamma_{g,1/2}^n\arrow[d,"{(s_F,p)}"]\arrow[r,"p"]&G_g\arrow[d, equal]\arrow[r]& 0\\
0\arrow[r]& H(g)\otimes \pi_n\SO\arrow[r]&(H(g) \otimes\pi_n\SO)\rtimes G_g\arrow[r]&G_g\arrow[r]& 0.
\end{tikzcd}
\end{equation}
The left vertical map is induced by the natural map $S\pi_n\SO(n)\ra \pi_n\SO$ originating from the inclusion $\SO(n)\subset\SO$ and is an isomorphism for $n\neq1,3,7$ odd as a consequence of the following lemma whose proof is standard (see e.g.\,\cite[§1B)]{Levine}).

\begin{lem}\label{lemma:stabiliseorthogonalgroups}For $n$ odd, the morphism $S\pi_n\SO(n)\ra \pi_n\SO$ induced by the inclusion $\SO(n)\subset\SO$ is an isomorphism for $n\neq1,3,7$ odd. For $n=1,3,7$, it is injective with cokernel $\bfZ/2$.
\end{lem}

As a result, the diagram \eqref{equation:stableframing} induces a splitting of \eqref{equation:Crowleyextension} for $n\neq3,7$ odd, since all vertical maps are isomorphisms. This proves the cases $n\neq3,7$ of the following reformulation of \cref{bigthm:splitextension} (see \cref{section:krecksextensions}). We postpone the proof of the cases $n=3,7$ to \cref{section:proofmainthm}.

\begin{thm}\label{thm:splitextension2}For $n\ge3$ odd, the extension
\[0\lra H(g)\otimes S\pi_n\SO(n)\lra \Gamma^n_{g,1/2}\xlra{p} G_g\lra0\] splits for $n\neq3,7$. For $n=3,7$, it splits if and only if $g=1$.
\end{thm}

Even though the extension does not split for $n=3,7$, the morphism \eqref{equation:stableframing} is still injective by \cref{lemma:stabiliseorthogonalgroups} and thus expresses the extension in question as a subextension of the trivial extension of $G_g$ by the $G_g$-module $H(g)\otimes \pi_n\SO$. Crowley \cite[Cor.\,3.5]{Crowley} gave an algebraic description of this subextension and concluded that it splits if and only if $g=1$. We proceed differently and prove this fact in \cref{section:extensionclass}
directly, which can in turn be used to determine the extension in the following way: by the discussion in \cref{section:stabilisation}, it is sufficient to determine its extension class in $\oH^2(G_g;H(g)\otimes S\pi_n\SO(n))$ for $g\gg0$. For $n=3,7$, we have $G_g\cong\Sp_{2g}(\bfZ)$ with its usual action on $H(g)\otimes S\pi_n\SO(n)\cong\bfZ^{2g}$. Using work of Djament \cite[Thm 1]{Djament}, one can compute $\oH^2(\Sp_{2g}(\bfZ);\bfZ^{2g})\cong \bfZ/2$ for $g\gg0$, so there is only one nontrivial extension of $G_g$ by $H(g)\otimes S\pi_n\SO(n)$, which must be the one in consideration because of the second part of \cref{thm:splitextension2}. Note that this line of argument gives a geometric proof for the following useful fact on the twisted cohomology of $\Sp_{2g}(\bfZ)$ as a byproduct, which can also be derived algebraically (see for instance \cite[Sect.\,2]{Crowley}).

\begin{cor}The pullback of the unique nontrivial class in $\oH^2(\Sp_{2g}(\bfZ);\bfZ^{2g})$ for $g\gg0$ to $\oH^2(\Sp_{2h}(\bfZ);\bfZ^{2h})$ for $h\le g$ is trivial if and only if $h=1$.
\end{cor}

We close this section by relating the abelianisation of $\Gamma^n_{g,1/2}$ to that of $G_g$. The latter is content of \cref{lemma:abelianisationSp}.

\begin{cor}\label{corollary:abelianisationhalfmcg}
For $n\ge3$ odd, the morphism \[\oH_1(\Gamma^n_{g,1/2})\lra\oH_1(G_g)\] is split surjective and has the coinvariants $(H(g)\otimes S\pi_n\SO(n))_{G_g}$ as its kernel, which vanish for $g\ge2$. For $g=1$ it vanishes if and only if $n\equiv5\Mod{8}$ or $n=3,7$, and has order $2$ otherwise.
\end{cor}
\begin{proof}
The claim regarding the coinvariants follows from \cref{lemma:coinvariantssymplectic} and \cref{table:SO}. Since they vanish for $g\ge2$, the remaining statement follows from the exact sequence
\[\oH_2(\Gamma^n_{g,1/2})\ra\oH_2(G_g)\ra \big(H(g)\otimes S\pi_n\SO(n)\big)_{G_g}\lra\oH_1(\Gamma^n_{g,1/2})\lra\oH_1(G_g)\lra0\] induced by  \eqref{equation:Crowleyextension}, combined with the fact that this extension splits for $g=1$ by \cref{thm:splitextension2}.
\end{proof}

\section{Signatures, obstructions, and \cref{bigthm:mainextension}}\label{section:mainsection}
By \cref{lemma:central}, the extension 
\[0\lra\Theta_{2n+1}\lra\Gamma^n_{g,1}\lra\Gamma^n_{g,1/2}\lra0\] discussed in \cref{section:krecksextensions} is central and is as such classified by a class in $\oH^2(\Gamma^n_{g,1/2};\Theta_{2n+1})$ with $\Gamma^n_{g,1/2}$ acting trivially on $\Theta_{2n+1}$. In this section, we identify this extension class in terms of the algebraic description of $\Gamma^n_{g,1/2}$ provided in the previous section, leading to a proof of our main result \cref{bigthm:mainextension}. Our approach is partially based on ideas of Galatius--Randal-Williams \cite[Sect.\,7]{GRWabelian}, who determined the extension for $n\equiv5\Mod{8}$ and $g\ge5$ up to automorphisms of $\Theta_{2n+1}$. 

We begin with an elementary recollection on the relation between Pontryagin classes and obstructions to extending trivialisations of vector bundles, mainly to fix notation.

\subsection{Obstructions and Pontryagin classes}\label{section:obstructionclass}
Let $k\ge1$ and $\xi\colon X\ra\tau_{>4k-1}\BO$ be a map to the $(4k-1)$st connected cover of $\BO$ with a lift $\bar{\xi}\colon A\ra\tau_{>4k}\BO$ over a subspace $A\subset X$ along the canonical map $\tau_{>4k}\BO\ra \tau_{>4k-1}\BO$. Such data has a relative Pontryagin class $p_k(\xi,\bar{\xi})\in\oH^{4k}(X,A;\bfZ)$ given as the pullback along the map $(\xi,\bar{\xi})\colon (X,A)\ra(\tau_{>4k-1}\BO,\tau_{>4k}\BO)$ of the unique lift to $\oH^{4k}(\tau_{>4k-1}\BO,\tau_{>4k}\BO;\bfZ)$ of the pullback $p_k\in\oH^{4k}(\tau_{>4k-1}\BO;\bfZ)$ of the usual Pontryagin class $p_k\in\oH^{4k}(\BSO;\bfZ)$. The class $p_k(\xi,\bar{\xi})$ is related to the primary obstruction $\chi(\xi,\bar{\xi})\in \oH^{4k}(X,A;\pi_{4k-1}\SO)$ to solving the lifting problem
\begin{center}
\begin{tikzcd}
A\arrow[d]\arrow[r,"\bar{\xi}"]&\tau_{>4k}\BO\arrow[d]\\
X\arrow[r,"\xi",swap]\arrow[ur,dashed]&\tau_{>4k-1}\BO
\end{tikzcd}
\end{center}
by the equality\[p_k(\xi,\bar{\xi})=\pm a_k(2k-1)!\cdot \chi(\xi,\bar{\xi}),\] up to the choice of a generator $\pi_{4k-1}\SO\cong \bfZ$ (cf.\,\cite[Lem.\,2]{MilnorKervaireBernoulli}). We suppress the lift $\bar{\xi}$ from the notation whenever there is no source of confusing. For us, $X=M$ will usually be a compact oriented $8k$-manifold and $A=\partial M$ its boundary, in which case we can evaluate $\chi^2(\xi,\bar{\xi})\in\oH^{8k}(M,\partial M;\bfZ)$ against the relative fundamental class $[M,\partial M]$ to obtain a number $\chi^2(\xi,\bar{\xi})\in\bfZ$. The following two sources of manifolds are relevant for us.

\begin{ex}\label{example:chiexamples}Fix an integer $n\equiv 3\Mod{4}$.
\begin{enumerate}
\item For a compact oriented $n$-connected $(2n+2)$-manifold whose boundary is a homotopy sphere, there is a (up to homotopy) unique lift $M\ra\tau_{>n}\BO$ of the stable oriented normal bundle. On the boundary $\partial M$, this lifts uniquely further to $\tau_{>n+1}\BO$, so we obtain a canonical class $\chi(M)\in\oH^{n+1}(M,\partial M;\bfZ)$ and a characteristic number $\chi^2(M)\in\bfZ$.
\item Consider a $(W_{g,1},D^{2n-1})$-bundle $\pi\colon E\ra B$, i.e.\,a smooth $W_{g,1}$-bundle with a trivialised $D^{2n-1}$-subbundle of its $\partial W_{g,1}$-bundle $\partial\pi\colon \partial E\ra B$ of boundaries. The standard framing of $D^{2n-1}$ induces a trivialisation of stable vertical tangent bundle $T_\pi E\colon E\ra \BSO$ over the subbundle $B\times D^{2n-1}\subset \partial E$, which extends uniquely to a $\tau_{>n+1}\BO$-structure on $T_\pi E|_{\partial E}$ by obstruction theory. Using that $W_g$ is $n$-parallelisable, another application of obstruction theory shows that the induced $\tau_{>n}\BO$-structure on $T_\pi E|_{\partial E}$ extends uniquely to a $\tau_{>n}\BO$-structure on $T_\pi E$, so the above discussion provides a class $\chi(T_\pi E)\in \oH^{n+1}(E,\partial E;\bfZ)$, and, assuming $B$ is an oriented closed surface, a number $\chi^2(T_\pi E)\in\bfZ$.
\end{enumerate}
\end{ex}

\subsection{Highly connected almost closed manifolds}\label{section:highlyconnected}As a consequence of \cref{thm:almostclosedbundles}, we shall see that $(W_{g,1},D^{2n-1})$-bundles over surfaces are closely connected to $n$-connected almost closed $(2n+2)$-manifolds. These manifolds were classified by Wall \cite{Wall2n}, which we now recall for $n\ge3$ in a form tailored to later applications, partly following \cite[Sect.\,2]{KrannichReinhold}.

A compact manifold $M$ is \emph{almost closed} if its boundary is a homotopy sphere. We write $A_{d}^{\tau_{>n}}$ for the abelian group of almost closed oriented $n$-connected $d$-manifolds up to oriented $n$-connected bordism restricting to an $h$-cobordism on the boundary. Recall that $\Omega_{d}^{\tau_{>n}}$ denotes the bordism group of closed $d$-manifolds $M$ equipped with a \emph{$\tau_{>n}\BO$-structure} on their stable normal bundle $M\ra\BO$, i.e.\,a lift $M\ra\tau_{>n}\BO$ to the $n$-connected cover. By classical surgery, the group $\Omega_{d}^{\tau_{>n}}$ is canonically isomorphic to the bordism group of closed oriented $n$-connected $d$-manifolds up to $n$-connected bordism as long as $d\ge2n+1$, so we will use both descriptions interchangeably. There is an exact sequence \begin{equation}\label{Wallexactsequence}\Theta_{2n+2}\lra\Omega_{2n+2}^{\tau_{>n}}\lra A_{2n+2}^{\tau_{>n}}\xlra{\partial}\Theta_{2n+1}\lra\Omega_{2n+1}^{\tau_{>n}}\lra 0\end{equation} due to Wall \cite[p.\,293]{Wall2n+1} in which the two outer morphisms are the obvious ones, noting that homotopy $d$-spheres $n$-connected for $n<d$. The morphisms $\Omega_{2n+2}^{\tau_{>n}}\ra A_{2n+2}^{\tau_{>n}}$ and $\partial\colon A_{2n+2}^{\tau_{>n}}\ra\Theta_{2n+1}$ are given by cutting out an embedded disc and by assigning to an almost closed manifold its boundary, respectively. By surgery theory, the subgroup \[\bA_{2n+2}\coloneq \im(A_{2n+2}^{\tau_{>n}}\xlra{\partial}\Theta_{2n+1})\] of homotopy $(2n+1)$-spheres bounding $n$-connected manifolds contains the cyclic subgroup $\bP_{2n+2}\subset \Theta_{2n+2}$ of homotopy $(2n+1)$-spheres bounding parallelisable manifolds, so the right end of \eqref{Wallexactsequence} receives canonical a map from Kervaire--Milnor's exact sequence \cite{KervaireMilnor}, 
\begin{equation}\label{equation:bAbP}
\begin{tikzcd}
0\arrow[r]&\bP_{2n+2}\arrow[r]\arrow[d,hook]&\Theta_{2n+1}\arrow[d,equal]\arrow[r]&\coker(J)_{2n+1}\arrow[r]\arrow[d,twoheadrightarrow]&0\\
0\arrow[r]&\bA_{2n+2}\arrow[r]&\Theta_{2n+1}\arrow[r]&\Omega^{\tau_{>n}}_{2n+1}\arrow[r]&0,
\end{tikzcd}
\end{equation}
which in particular induces a morphism $\coker(J)_{2n+1}\ra\Omega^{\tau_{>n}}_{2n+1}$, concretely given by representing a class in $\coker(J)_{2n+1}$ by a stably framed manifold and restricting its stable framing to a $\tau_{>n}\BO$-structure.

\subsubsection{Wall's classification}\label{section:Wallclassification}For our purposes, Wall's computation \cite{Wall2n,Wall2n+1} of $A_{2n+2}^{\tau_{>n}}$ is for $n\ge3$ odd is most conveniently stated in terms of two particular almost closed $n$-connected $(2n+2)$-manifolds, namely
\begin{enumerate}\item \emph{Milnor's $E_8$-plumbing} $P$, arising from plumbing together $8$ copies of the disc bundle of the tangent bundle of the standard $(n+1)$-sphere such that the intersection form of $P$ agrees with the $E_8$-form (see e.g.\cite[Ch.\,V.2]{BrowderBook}), and
\item the manifold $Q$, obtained from plumbing together two copies of a linear $D^{n+1}$-bundle over the $(n+1)$-sphere representing a generator of $S\pi_n\SO(n)$. 
\end{enumerate}
The following can be derived from Wall's work, as explained for instance in \cite[Thm\,2.1]{KrannichReinhold}.

\begin{thm}[Wall]\label{thm:Wall}
For $n\ge3$ odd, the bordism group $A_{2n+2}^{\tau_{>n}}$ satisfies
\[A_{2n+2}^{\tau_{>n}}\cong 
\begin{cases}
\bfZ\oplus \bfZ/2&\mbox{if }n\equiv1\Mod{8}\\
\bfZ\oplus \bfZ&\mbox{if }n\equiv3\Mod{4}\\
\bfZ&\mbox{if }n\equiv5\Mod{8}.
\end{cases}\] The first summand is generated by the class of $P$ in all cases but $n=3,7$ in which it is generated by $\bfH P^2$ and $\bfO P^2$. The second summand for $n\not{equiv}5\Mod{8}$ is generated by $Q$.
\end{thm}

From a consultation of \cref{table:SO}, one sees that the group $S\pi_n\SO(n)$ vanishes for $n\equiv5\Mod{8}$, so $Q\in A_{2n+2}^{\tau_{>n}}$ is trivial in this case, which shows that the subgroup $\bA_{2n+2}$ is for all $n\ge3$ odd generated by the boundaries 
\[\Sigma_P\coloneq \partial P\in\bA_{2n+2}\quad\text{and}\quad\Sigma_Q\coloneq \partial Q\in\bA_{2n+3}.\]
In the cases $n\equiv 1\Mod{8}$ in which $Q$ defines a $\bfZ/2$-summand, its boundary $\Sigma_Q$ is trivial by a result of Schultz \cite[Cor.\,3.2, Thm\,3.4 iii)]{Schultz}. For $n\equiv3\Mod{4}$ on the other hand, it is nontrivial by a calculation of Kosinski \cite[p.\,238-239]{Kosinski}.

\begin{thm}[Kosinski, Schultz]\label{thm:Schultz}The homotopy sphere $\Sigma_Q\in\Theta_{2n+1}$ is trivial for $n\equiv 1\Mod{4}$ and nontrivial for $n\equiv 3\Mod{4}$
\end{thm}

In the exceptional dimensions $n=3,7$, the homotopy sphere $\Sigma_Q$ agrees with the inverse of the Milnor sphere $\Sigma_P$, as explained in \cite[Cor.\,2.8 ii)]{KrannichReinhold}.

\begin{lem}\label{lemma:sigmaqsigmap}$\Sigma_Q=-\Sigma_P$ for $n=3,7$.
\end{lem}

For $n\ge3$ odd, the Milnor sphere $\Sigma_P\in\Theta_{2n+1}$ is well-known to be nontrival and to generate the cyclic subgroup $\bP_{2n+2}\subset\Theta_{2n+1}$ whose order can be expressed in terms of numerators of divided Bernoulli numbers (see e.g.\,\cite[Lem.\,3.5 (2), Cor.\,3.20]{Levine}), so \cref{thm:Schultz} has the following corollary.

\begin{cor}\label{corollary:generatorsbA}For $n\ge3$ odd, the subgroup $\bA_{2n+1}$ is nontrivial. It is generated by $\Sigma_P$ for $n\equiv1\Mod{4}$, by $\Sigma_Q$ for $n=3,7$, and by $\Sigma_P$ and $\Sigma_Q$ for $n\equiv3\Mod{4}$.
\end{cor}

Combining the previous results with the diagram \eqref{equation:bAbP}, we obtain the following result, which we already mentioned in the introduction.

\begin{cor}\label{corollary:SchultzWall}The natural morphism $\coker(J)_{2n+1}\rightarrow\Omega^{\tau_{>n}}_{2n+2}$ is an isomorphism for $n\equiv1\Mod{4}$ and for $n=3,7$. For $n\equiv3\Mod{4}$, it is an epimorphism whose kernel is generated by the class $[\Sigma_Q]\in\coker(J)_{2n+1}$.
\end{cor}

$[\Sigma_Q]\in\coker(J)_{2n+1}$ is conjecturally trivial \cite[Conj.\,A]{GRWabelian} for all $n$ odd. Until recently (see \cref{rem:BHS}), this was only known for $n=3,7$ and $n\equiv1\Mod{4}$ by the results above.
\begin{conj}[Galatius--Randal-Williams]\label{conjecture:GRW}$[\Sigma_Q]=0$ for all $n\ge3$ odd.
\end{conj}

\begin{rem}\label{rem:BHS}As mentioned in the introduction, after the completion of this work, Burklund--Hahn--Senger \cite{BHS} and Burklund--Senger \cite{BS} showed that $[\Sigma_Q]$ vanishes in $\coker(J)_{2n+1}$ for $n$ odd if and only if $n\neq 11$, confirming \cref{conjecture:GRW} for $n\ne11$ and disproving it for $n=11$. This has as a consequence that, for $n\neq 11$ odd, the subgroup $\bA_{2n+1}$ is generated by $\Sigma_P$ even for $n\equiv 3\Mod{4}$ and that the morphism $\coker(J)_{2n+1}\ra \Omega_{2n+2}^{\tau_{>n}}$ discussed in \cref{corollary:SchultzWall} is an isomorphism.
\end{rem}

\subsubsection{Invariants}It follows from \cref{thm:Wall} and \cref{thm:Schultz} that the boundary of an $n$-connected almost closed $(2n+2)$-manifold $M$ is determined by at most two integral bordism invariants of $M$. Concretely, we consider the signature $\sgn\colon A_{2n+2}^{\tau_{>n}}\ra\bfZ$ and for $n\equiv 3\Mod{4}$ the characteristic number $\chi^2\colon A_{2n+2}^{\tau_{>n}}\ra\bfZ$, explained in \cref{example:chiexamples}. As discussed for example in \cite[Sect.\,2.1]{KrannichReinhold}, these functionals evaluate to \begin{equation}\label{equation:numbersalmostclosed}\sgn(P)=8 \quad \quad\sgn(Q)=0\quad \quad\chi^2(P)=0\quad \quad\chi^2(Q)=\begin{cases}8&\text{for }n=3,7\\2&\text{otherwise}\end{cases},\end{equation} and on the closed manifolds $\bfH P^2$ and $\bfO P^2$ to
\begin{equation}\label{equation:exceptionalnumbersalmostclosed}\sgn(\bfH P^2)=\sgn(\bfO P^2)=1\quad \quad \chi^2(\bfH P^2)=\chi^2(\bfO P^2)=1,\end{equation} which results in the following formula for boundary spheres of highly connected manifolds when combined with the discussion above.

\begin{prop}\label{prop:computingboundaries}
For $n\ge 3$ odd, the boundary $\partial M\in \Theta_{2n+1}$ of an almost closed oriented $n$-connected $(2n+2)$-manifold $M$ satisfies
\[\partial M=
\begin{cases}
\sgn(M)/8\cdot\Sigma_P&\mbox{if }n\equiv1\Mod{4}\\
\sgn(M)/8\cdot\Sigma_P+\chi^2(M)/2\cdot\Sigma_Q&\mbox{if }n\equiv3\Mod{4}\text{ and }n\neq 3,7\\
\big(\chi^2(M)-\sgn(M)\big)/8\cdot\Sigma_Q&\mbox{if }n=3,7.
\end{cases}\]
\end{prop}

\subsubsection{The minimal signature}\label{section:minimalsignature} As in the introduction, we denote by $\sigma_n'$ the minimal positive signature of a smooth closed $n$-connected $(2n+2)$-manifold. This satisfies $\sigma_n'=1$ for $n=1,3,7$ as witnessed by $\bfC P^2$, $\bfH P^2$, and $\bfO P^2$, and in all other cases, it can be expressed in terms of the subgroup $\bA_{2n+2}\subset \Theta_{2n+1}$ as follows.

\begin{lem}\label{lemma:quotientbySigmaQ}For $n\ge3$ odd, the quotient $\bA_{2n+2}/\langle \Sigma_Q \rangle $ is a cyclic group generated by the class of $\Sigma_P$. It is trivial if $n=3,7$ and of order $\sigma_n'/8$ otherwise.\end{lem} 

\begin{proof}For $n=3,7$, the claim is a direct consequence of \cref{thm:Wall} and \cref{lemma:sigmaqsigmap} and in the case $n\neq3,7$, it follows from taking vertical cokernels in the commutative diagram
\begin{center}
\begin{tikzcd}
0\arrow[r]&\langle \ord(\Sigma_Q)\cdot Q \rangle\arrow[r]\arrow[d,hook]&\Omega_{2n+2}^{\tau_{>n}}/\Theta_{2n+1}\arrow[r,"\sgn"]\arrow[d,hook]&\sigma_n'\cdot\bfZ\arrow[r]\arrow[d,hook]&0\\
0\arrow[r]&\langle Q \rangle\arrow[r]&A_{2n+2}^{\tau_{>n}}\arrow[r,"\sgn"]&8\cdot\bfZ\arrow[r]&0
\end{tikzcd}
\end{center}
with exact rows, obtained from a combination of \cref{thm:Wall} with  \eqref{Wallexactsequence} and \eqref{equation:numbersalmostclosed}.
\end{proof}

\begin{rem}\label{remark:minimalsignature}In \cite[Prop.\,2.15]{KrannichReinhold}, the minimal signature $\sigma_n'$ was expressed in terms of Bernoulli numbers and the order of $[\Sigma_Q]\in\coker(J)_{2n+1}$, from which one can conclude that for $n\neq1,3,7$, the signature of such manifolds is divisible by $2^{n+3}$ if $(n+1)/2$ is odd and by $2^{n-2\nu_2(n+1)}$ otherwise, where $\nu_2(-)$ denotes the $2$-adic valuation (see \cite[Cor.\,2.18]{KrannichReinhold}).
\end{rem}

\subsection{Bundles over surfaces and almost closed manifolds} \label{section:extensionclass}
In order to identify the cohomology class in $\oH^2(\Gamma_{g,1/2}^n;\Theta_{2n+1})$ that classifies the central extension
\begin{equation}\label{equation:extension999}0\lra\Theta_{2n+1}\lra\Gamma^n_{g,1}\lra\Gamma^n_{g,1/2}\lra0,\end{equation}
we first determine how it evaluates against homology classes, i.e.\,identify its image 
\[d_2\colon\oH_2(\Gamma^n_{g,1/2};\bfZ)\lra \Theta_{2n+1}\]
under the map $h$ participating in the universal coefficient theorem
\[0\ra\Ext(\oH_1(\Gamma^n_{g,1/2};\bfZ),\Theta_{2n+1})\ra\oH^2(\Gamma^n_{g,1/2};\Theta_{2n+1})\xlra{h}\Hom(\oH_2(\Gamma^n_{g,1/2};\bfZ);\Theta_{2n+1})\ra 0,\] followed by resolving the remaining ambiguity originating from the Ext-term. Indicated by our choice of notation, the morphism $d_2$ can be viewed alternatively as the first possibly nontrivial differential in the $E_2$-page of the Serre spectral sequence of the extension \eqref{equation:extension999} (cf.\,\cite[Thm 4]{HochschildSerre}). Before identifying this differential, we remind the reader of two standard facts we shall make frequent use of.
\begin{enumerate}
\item The natural map $\MSO\ra\HZ$ is $4$-connected, so pushing forward fundamental classes induces an isomorphism $\Omega^{\SO}_*(X)\ra \oH_*(X;\bfZ)$ for $*\le3$ and any space $X$, and
\item the $1$-truncation of a connected space $X$ (in particular the natural map $\oB G\ra\oB\pi_0G$ for a topological group $G$) induces a surjection $\oH_2(X;\bfZ) \twoheadrightarrow\oH_2(K(\pi_1X,1);\bfZ)$, whose kernel agrees with the image of the Hurewicz homomorphism $\pi_2X\ra\oH_2(X;\bfZ)$.
\end{enumerate}

The key geometric ingredient to identify the differential $d_2$ is the following result.

\begin{thm}\label{thm:almostclosedbundles}
Let $n\ge3$ be odd and $\pi\colon E\ra S$ a $(W_{g,1},D^{2n-1})$-bundle over an oriented closed surface $S$. There exists a class $E'\in A_{2n+2}^{\tau_{>n}}$ such that
\begin{enumerate}
\item its boundary $\partial E'\in\Theta_{2n+1}$ is the image of the class $[\pi]\in\oH^2(\BDiffuo_{\partial/2}(W_{g,1});\bfZ)$ under
\[\oH_2(\BDiffuo_{\partial/2}(W_{g,1});\bfZ)\longtwoheadrightarrow\oH_2(\Gamma^n_{g,1/2};\bfZ)\xlra{d_2}\Theta_{2n+1},\]
\item it satisfies $\sgn(E)=\sgn(E')$, and $\chi^2(T_\pi E)=\chi^2(E')$ if $n\equiv3\Mod{4}.$
\end{enumerate}
\end{thm}
\begin{proof}
By the isotopy extension theorem, the restriction map to the moving part of the boundary $\Diffuo_{\partial/2}(W_{g,1})\ra\Diffuo_\partial(D^{2n-1})$ is a fibration. As its image is contained in the component of the identity (see the proof of \cref{lemma:central}), this fibration induces the upper row in a map of fibre sequences
\begin{equation}
\begin{tikzcd}\label{equ:fibremapgeometricthm}
\Diffuo^{\id}_\partial(D^{2n-1})\rar\dar&\BDiffuo_{\partial}(W_{g,1})\rar\dar& \BDiffuo_{\partial/2}(W_{g,1})\dar\\\oB\Theta_{2n+1}\rar&\oB\Gamma^n_{g,1}\rar&\oB\Gamma^n_{g,1/2}
\end{tikzcd}
\end{equation}
whose bottom row is induced by the extension \eqref{equation:extension999}. The two right vertical maps are induced by taking components and the left vertical map is the induced map on homotopy fibres. The latter agrees with the delooping of the Gromoll map $\Omega\Diffuo_\partial(D^{2n-1})\ra \Diffuo_\partial(D^{2n})$ followed by taking components, which one checks by looping the fibre sequences and using that\[\Omega\Diffuo^{\id}_\partial(D^{2n-1})\lra\Diffuo_{\partial}(W_{g,1})\] is given by ``twisting'' a collar $ [0,1]\times S^{2n-1}\subset W_{g,1}$, meaning that it sends a smooth loop $\gamma\in\Omega \Diffuo^{\id}_\partial(D^{2n-1})\subset \Omega\Diffuo^{\id}_\partial(S^{2n-1})$ to the diffeomorphism that is the identity outside the collar and is given by $(t,x)\mapsto (t,\gamma(t)\cdot x)$ on the collar. Now consider the square
\begin{equation}\label{equ:squaretransgression}
\begin{tikzcd}
\oH_2(\BDiffuo_{\partial/2}(W_{g,1});\bfZ)\arrow[d]\arrow[r]&\oH_2(\BDiffuo_{\partial}^{\id}(D^{2n-1});\bfZ)\arrow[d]\\
\oH_2(\Gamma^n_{g,1/2};\bfZ)\arrow[r]&\Theta_{2n+1}
\end{tikzcd}
\end{equation} 
obtained from delooping \eqref{equ:fibremapgeometricthm} once to the right and using $\oH_2(B^2\Theta_{2n+1})\cong\Theta_{2n+1}$. By transgression, the bottom horizontal arrow agrees with the differential $d_2$ in the statement.
Combining this with the Hurewicz theorem, the square \eqref{equ:squaretransgression} provides a factorisation 
\[\oH_2(\BDiffuo_{\partial/2}(W_{g,1});\bfZ)\ra \oH_2(\BDiffuo_{\partial}^{\id}(D^{2n-1});\bfZ)\cong\pi_2(\BDiffuo_{\partial}^{\id}(D^{2n-1}))\lra \Theta_{2n+1}\] of the map in the first part of the statement, which thus has the following geometric description: a smooth $(W_{g,1},D^{2n-1})$-bundle $\pi\colon E\ra S$ represents a class $[\pi]\in\oH^2(\BDiffuo_{\partial/2}(W_{g,1});\bfZ)$ and its image under the first map in the composition is the class $[\pi_+]\in\oH_2(\BDiffuo^{\id}_{\partial}(D^{2n-1});\bfZ)$ of its $(\partial W_{g,1},D^{2n-1})$-bundle $\pi_+\colon E_+\ra S$ of boundaries, which in turn maps under the inverse of the Hurewicz homomorphism to a $(\partial W_{g,1},D^{2n-1})$-bundle $\pi_-\colon E_-\ra S^2$ over the $2$-sphere that is bordant, as a bundle, to $\pi_-$. That is, there exists a $(\partial W_{g,1},D^{2n-1})$-bundle $\bar{\pi}\colon \bar{E}\ra K$ over an oriented bordism $K$ between $S$ and $S^2$ that restricts to $\pi_+$ over $S$ and to $\pi_-$ over $S^2$. We claim that the image of the $(\partial W_{g,1},D^{2n-1})$-bundle $\pi_-$ under the final map in the composition is the homotopy sphere $\Sigma_\pi\in\Theta_{2n+1}$ obtained by doing surgery on the total space $E_-$ along the trivialised subbundle $D^{2n-1}\times S^2\subset E_-$. This is most easily seen by thinking of a class in $\pi_k\BDiffuo_\partial(D^d)$ as a smooth bundle $D^d\ra P\ra D^k$ together with a trivialisation $\varphi\colon D^d\times \partial D^k\cong P|_{\partial D^k}$ and a trivialised $\partial D^k$-subbundle $\psi\colon \partial D^d\times D^k\hookrightarrow P$ such that $\varphi$ and $\psi$ agree on $\partial D^d\times \partial D^k$. From this point of view, the morphism $\pi_k\BDiffuo_\partial(D^d)\ra \pi_{k-1}\BDiffuo_\partial(D^{d}\times D^1)\cong \pi_{k-1}\BDiffuo_\partial(D^{d+1})$ induced by the Gromoll map is given by sending such a bundle $p\colon P\ra D^k\cong D^{k-1}\times D^1$ to the $(D^d\times D^1)$-bundle $(\mathrm{pr}_1\circ p)\colon P\ra D^{k-1}$, and the isomorphism $\pi_1\BDiffuo_\partial(D^{d})\cong \Theta_{d+1}$ is given by assigning to a disc bundle $D^d\ra P\ra D^1$ the manifold $P\cup_{\partial D^d\times D^1\cup D^d\times\partial D^1} D^d\times D^1$. Deccomposing the sphere into half-discs $S^n=D^n_+\cup D^n_-$, we see from this description that the composition $\pi_k\BDiffuo_\partial(D^d)\ra \Theta_{d+k}$ of the iterated Gromoll map with the isomorphism $\pi_1\BDiffuo_\partial(D^{d+k-1})\cong \Theta_{d+k}$ maps a class represented by an $S^{d}$-bundle $S^d\ra Q\ra S^k$ with a trivialisation $\varphi\colon D_+^k\times S^d\cong Q|_{D_+^k}$ and a trivialised subbundle $\psi\colon S^k\times D_+^d\hookrightarrow Q$ that agree on $D^k_+\times D^d_+$ to the homotopy sphere
\[\begin{gathered}
\Big(Q\backslash \mathrm{int}\big(\varphi (D_+^k\times S^d)\cup \psi(S^k\times D_+^d)\big)\Big)\cup (D^k\times D^d)\cong \Big(Q\backslash\mathrm{int}\big(\psi(S^k\times D^d_+)\big)\Big)\cup_{S^k\times \partial D^d}(D^{k+1}\times \partial D^d)
\end{gathered}
\]
obtained by doing surgery along the trivialised $D^d$-subbundle, where $D^k\times D^d$ is glued to $Q\backslash \mathrm{int}\big(\varphi (D_+^k\times S^d)\cup_{\partial(D^k\times S^d)} \psi(S^k\times D_+^d)\big)$ along the embedding
\[\partial(D^k\times D^d)=\partial D^k\times D^d\cup D^k\times \partial D^d\xra{\varphi|_{\partial D^k_+\times D^d_-}\cup \psi|_{D^k_-\times \partial D^d_+}}\partial\Big(Q\backslash \mathrm{int}\big(\varphi (D_+^k\times S^d)\cup \psi(S^k\times D_+^d)\Big).\]
 This in particular implies the claim we made above in the case $k=2$ and $d=2n-1$.

As a consequence of this description of the morphism in consideration, the image $\Sigma_\pi\in \Theta_{2n+1}$ of the class $[\pi]$ comes equipped with a nullbordism, namely $N\coloneq E\cup_{E_+}\bar{E}\cup_{E_-}W,$ where $W$ is the trace of the performed surgery. Omitting the trivialised $D^{2n-1}$-subbundles, the situation can be summarised schematically as follows
\begin{center}
\begin{tikzcd}[column sep=0.4cm, row sep=0.4cm]
W_{g,1}\arrow[d]&\partial W_{g,1}\arrow[d]&\partial W_{g,1}\arrow[d]&\partial W_{g,1}\arrow[d]&\phantom{\partial W_{g,1}}&\\
E\arrow[d,"\pi"]\arrow[r,symbol=\supset]&E_+\arrow[d,"\pi_+"]\arrow[r,symbol=\subset]&\bar{E}\arrow[r,symbol=\supset]\arrow[d,"\bar{\pi}"]&E_-\arrow[d,"\pi_-"]\arrow[r,symbol=\subset]&W\arrow[r,symbol=\supset]&\Sigma_\pi\\
S\arrow[r,symbol={=}]&S\arrow[r,symbol=\subset]&K\arrow[r,symbol=\supset]&S^2.&&&
\end{tikzcd}
\end{center}

A choice of a stable framing of $K$ induces stable framings on $S$ and $S^2$ and thus a stable isomorphism $TE\cong T_\pi E\oplus\pi^*TS\cong_sT_\pi E$ using which the canonical $\tau_{>n}\BO$-structure on $T_\pi E$ and the $\tau_{>n+1}\BO$-one on $T_\pi E|_{E_+}$ (see \cref{example:chiexamples}) induce a $\tau_{>n}\BO$-structure on $TE$ and a $\tau_{>n+1}\BO$-structure on $TE|_{E_+}\cong_sTE_+$. With these choices, we have $\chi(T_\pi E)=\chi(TE,TE_+)$. By construction, the restriction of this $\tau_{>n+1}\BO$-structure to $TE_+|_{S\times D^{2n-1}}\cong_s TS$ agrees with the $\tau_{>n+1}\BO$-structure on $TS$ obtained from the stable framing of $K$, so we obtain a $\tau_{>n+1}\BO$-structure on $T\bar{E}|_{E_+\cup K\times D^{2n-1}}$, which by obstruction theory extends to one on $T(\bar{E}\cup_{E_-}W)$: the relative Serre spectral sequence shows that $H^*(\bar{E},E_+\cup K\times D^{2n-1})$ vanishes for $*\le 2n-2$ and thus that $H^{i+1}(\bar{E}\cup_{E_-}W,E_+\cup K\times D^{2n-1};\pi_i(\tau_{\le n}\SO))\cong H^{i+1}(\bar{E}\cup_{E_-}W,\bar{E};\pi_i(\tau_{\le n}\SO))\cong H^{i+1}(W,E_-;\pi_i(\tau_{\le n}SO))\cong H^{i+1}(D^3,S^2;\pi_i(\tau_{\le n}SO))=0$ for $i\le 2n-3$, using $W\simeq E_{-}\cup_{S^2}D^3$ and $\pi_2SO=0$. The restriction of this $\tau_{>n+1}\BO$-structure on $T\bar{E}|_{E_+\cup K\times D^{2n-1}}$ to a $\tau_{>n}\BO$ and the canonical $\tau_{>n}\BO$-structure on $TE\cong T_\pi E$ (see \cref{example:chiexamples}) assemble to a $\tau_{>n}\BO$-structure on $N$. By construction, the canonical restriction map (using excision)
\[H^*(E,E_+;\bfZ)\cong H^*(N,\bar{E}\cup_{E_-}W; \bfZ)\lra H^*(N,\Sigma_\pi; \bfZ)\]
sends $\chi(T_\pi E)=\chi(TE,TE_+)$ to $\chi(TN,T\Sigma_\pi)$, so we conclude $\chi^2(T_\pi E)=\chi^2(TN,T\Sigma_\pi)$. To finish the proof, note that the $\tau_{>n}\BO$-structure on $TN$ allows us to do surgery away from the boundary on $N$ to obtain an $n$-connected manifold $E'$, which gives a class in $A_{2n+2}^{\tau_{>n}}$ as aimed for: $\partial E'=\Sigma_\pi$ holds by construction,  $\chi^2(E')=\chi^2(TN,T\Sigma_\pi)=\chi^2(T_\pi E)$ by the bordism invariance of Pontryagin numbers (see \cref{example:chiexamples}), and $\sgn(E')=\sgn(N)=\sgn(E)$ by the additivity and bordism invariance of the signature.\end{proof}

Combining the previous result with \cref{prop:computingboundaries}, we conclude that the composition \[\oH_2(\BDiffuo_{\partial/2}(W_{g,1});\bfZ) \longtwoheadrightarrow\oH_2(\Gamma^n_{g,1/2};\bfZ)\xlra{d_2}\Theta_{2n+1}\] sends a homology class $[\pi]$ represented by a bundle $\pi\colon E\ra S$ to a certain linear combination of $\Sigma_P$ and $\Sigma_Q$ whose coefficients involve the invariants $\sgn(E)$ and $\chi^2(T_\pi E)$. In the following two subsections, we shall see that these functionals \[\sgn\colon \oH_2(\BDiffuo_{\partial/2}(W_{g,1});\bfZ)\lra\bfZ\quad\text{and}\quad\chi^2\colon \oH_2(\BDiffuo_{\partial/2}(W_{g,1});\bfZ)\lra\bfZ\] factor through the composition
\[\oH_2(\BDiffuo_{\partial/2}(W_{g,1});\bfZ)\longtwoheadrightarrow \oH_2(\Gamma^n_{g,1/2} ;\bfZ)\xlra{(s_F,p)_*}\oH_2((H(g)\otimes \pi_n\SO)\rtimes G_g ;\bfZ)\] and have a more algebraic description in terms of $H(g)\otimes \pi_n\SO)\rtimes G_g$. This uses the morphism \[(s_F,p)\colon \Gamma^n_{g,1/2}\lra (H(g)\otimes \pi_n\SO)\rtimes G_g,\] induced by acting on a stable framing of $W_{g,1}$ (agreeing with the usual framing on $D^{2n-1}$) as explained in \cref{section:firstextension}.

\subsection{Signatures of bundles of symplectic lattices}\label{section:algebraicsignature}The standard action of the symplectic group $\Sp_{2g}(\bfZ)$ on $\bfZ^{2g}$ gives rise to a local system $\cH(g)$ over $\BSp_{2g}(\bfZ)$ and the usual symplectic form on $\bfZ^{2g}$ gives a morphism $\lambda\colon\cH(g)\otimes \cH(g)\ra \bfZ$ of local systems to the constant system. To an oriented closed surface $S$ with a map $f\colon S\rightarrow \BSp_{2g}(\bfZ)$, we can associate a bilinear form \[\langle -,-\rangle_f\colon \oH^1(S; f^*\cH(g))\otimes \oH^1(S; f^*\cH(g))\rightarrow \bfZ,\] defined as the composition \[\oH^1(S; f^*\cH(g))\otimes \oH^1(S; f^*\cH(g))\xrightarrow{\smile} \oH^2(S; f^*\cH(g)\otimes \cH(g))\xlra{\lambda}\oH^2(S;\bfZ)\xrightarrow{[S]}\bfZ.\]  As both the cup-product and the symplectic pairing $\lambda$ are antisymmetric, the form $\langle -,-\rangle_f$ is symmetric. The usual argument for the bordism invariance of the signature shows that its signature $\sgn(\langle-,-\rangle_f)$ depends only on the bordism class $[f]\in\Omega^{\SO}_{2}(\BSp_{2g}(\bfZ))\cong\oH_2(\Sp_{2g}(\bfZ);\bfZ)$ and thus induces a morphism
\begin{equation}\label{signaturemorphism}\sgn\colon\oH_2(\Sp_{2g}(\bfZ);\bfZ)\lra\bfZ,\end{equation} which is compatible with the usual inclusion $\Sp_{2g}(\bfZ)\subset\Sp_{2g+2}(\bfZ)$.

\begin{rem}As $\Sp_{2g}(\bfZ)$ is perfect for $g\ge3$ (see \cref{lemma:abelianisationSp}), the morphism \eqref{signaturemorphism} determines a unique cohomology class $\sgn\in\oH^2(\Sp_{2g}(\bfZ);\bfZ)$. There is a well-known (purely algebraic) cocycle representative of this class due to Meyer \cite{Meyer}, known as the \emph{Meyer cocycle}.
\end{rem}

The morphism \eqref{signaturemorphism} measures signatures of total spaces of smooth bundles over surfaces (even of fibrations of Poincaré complexes). More precisely, for a compact oriented $(4k+2)$-manifold $M$, the action of its group of diffeomorphisms on the middle cohomology induces a morphism $\Diff(M)\ra\Sp_{2g}(\bfZ)$ for $2g=\mathrm{rk}(\oH^{2k+1}(M))$ and the resulting composition \[\oH_2(\BDiff(M) ;\bfZ)\lra\oH_2(\Sp_{2g}(\bfZ) ;\bfZ)\xlra{\sgn}\bfZ\] can be shown to map a homology class represented by a smooth bundle over a surface to the signature of its total space. This fact can either be proved along the lines of \cite{ChernHirzebruchSerre} or extracted from \cite{MeyerThesis} and it has in particular the following consequence.

\begin{lem}\label{lemma:signatureidentification}For $n$ odd, the composition
\[\oH_2(\BDiffuo_{\partial/2}(W_{g,1});\bfZ)\lra\oH_2(\Sp_{2g}(\bfZ);\bfZ)\xra{\sgn}\bfZ\]
sends the class of an $(W_{g,1},D^{2n-1})$-bundle $\pi\colon E\ra S$ to the signature $\sgn(E)$ of its total space.
\end{lem}

We proceed by computing the image of the signature morphism \eqref{signaturemorphism} and of its pullback to the theta-subgroup $\Sp_{2g}^q(\bfZ)\subset\Sp_{2g}(\bfZ)$ as defined in \cref{section:Wallsform}.

\begin{lem}\label{lemma:algebraicsignatures}The signature morphism satisfies
\begin{gather*}\im\left(\oH_2(\Sp_{2g}(\bfZ);\bfZ)\xra{\sgn}\bfZ\right)=\begin{cases}0&\mbox{if }g=1\\
4\cdot\bfZ&\mbox{if }g\ge2\\
\end{cases}\\\im\left(\oH_2(\Sp^q_{2g}(\bfZ);\bfZ)\xra{\sgn}\bfZ\right)=\begin{cases}0&\mbox{if }g=1\\
8\cdot\bfZ&\mbox{if }g\ge2.\\
\end{cases}\end{gather*}\end{lem}

\begin{proof}The signatures realised by classes in $\oH_2(\Sp_{2g}(\bfZ);\bfZ)$ are well-known (see e.g.\,\cite[Lem\,6.5, Thm\,6.6\,(vi)]{BensonCampagnoloRanickiRovi}). To prove that the signature of classes in $\oH_2(\Sp_{2g}^q(\bfZ))$ is divisible by $8$, recall from Sections~\ref{section:Wallsform} and~\ref{section:krecksextensions} that for $n$ odd the morphism $\Diffuo_{\partial/2}(W_{g,1})\ra \Sp_{2g}(\bfZ)$ lands in the subgroup $\Sp_{2g}^q(\bfZ)\subset\Sp_{2g}(\bfZ)$ as long as $n\neq1,3,7$, so we have a composition
\[\oH_2(\BDiffuo_{\partial/2}(W_{g,1});\bfZ)\lra\oH_2(\Gamma^n_{g,1/2};\bfZ)\lra\oH_2(\Sp_{2g}^q(\bfZ);\bfZ)\xra{\sgn}\bfZ,\] which maps the class of a bundle $\pi\colon E\ra S$ by \cref{lemma:signatureidentification} to $\sgn(E)$. The latter agrees by \cref{thm:almostclosedbundles} with the signature of an almost closed $n$-connected $(2n+2)$-manifold, so it is divisible by $8$ as the intersection form of such manifolds is unimodular and even (see e.g.\,\cite{Wall2n}). This proves the claimed divisibility, since the first two morphisms in the composition are surjective, the first one because of the second reminder at the beginning of \cref{section:extensionclass} and the second one by \cref{corollary:abelianisationhalfmcg}. As the signature morphism vanishes on $\oH_2(\Sp_{2}(\bfZ);\bfZ)$ by the first part, it certainly vanishes on $\oH_{2g}(\Sp^{q}_{2}(\bfZ);\bfZ)$. Consequently, by the compatibility of the signature with the inclusion $\Sp_{2g}(\bfZ)\subset\Sp_{2g+2}(\bfZ)$, the remaining claim follows from constructing a class in $\oH_2(\Sp_{2g}^q(\bfZ);\bfZ)$ of signature $8$ for $g=2$. Using $\oH_2(\Sp_{4}(\bfZ);\bfZ)\cong \bfZ\oplus\bfZ/2$ (see e.g.\,\cite[Lem.\,A.1(iii)]{BensonCampagnoloRanickiRovi}) and the first part of the claim, the existence of such a class is equivalent to the image of $\oH_2(\Sp^q_4(\bfZ);\bfZ)$ in the torsion free quotient $\oH_2(\Sp_{4}(\bfZ);\bfZ)_{\free}\cong\bfZ$ containing $2$. That it contains $10$ is ensured by transfer, since the index of $\Sp_{4}^q(\bfZ)\subset\Sp_{4}(\bfZ)$ is $10$ (see \cref{section:Wallsform}). As $\oH_1(\Sp_4(\bfZ);\bfZ)$ and $\oH_1(\Sp^q_4(\bfZ);\bfZ)$ are $2$-torsion by \cref{lemma:abelianisationSp}, it therefore suffices to show that $\oH_2(\Sp^q_4(\bfZ);\bfF_5)\ra\oH_2(\Sp_{4}(\bfZ);\bfF_5)$ is nontrivial, for which we consider the level $2$ congruence subgroup $\Sp_{4}(\bfZ,2)\subset\Sp_4(\bfZ)$, i.e.\,the kernel of the reduction map $\Sp_4(\bfZ)\ra\Sp_4(\bfZ/2)$, which is surjective (see e.g.\,\cite[Thm 1]{NewmanSmart} for an elementary proof). From the explicit description of $\Sp_{2g}^q(\bfZ)$ presented in \cref{section:Wallsform}, one sees that it contains the congruence subgroup $\Sp_{4}(\bfZ,2)$. As a result, it is enough to prove that $\oH_2(\Sp_4(\bfZ,2);\bfF_5)\ra\oH_2(\Sp_{4}(\bfZ);\bfF_5)$ is nontrivial, which follows from an application of the Serre spectral sequence of the extension
\[0\lra\Sp_4(\bfZ,2)\lra\Sp_4(\bfZ)\lra\Sp_4(\bfZ/2)\lra0,\] using that $\oH_1(\Sp_{4}(\bfZ,2);\bfF_5)$ vanishes by a result of Sato \cite[Cor.\,10.2]{Sato} and that the groups $\oH_*(\Sp_4(\bfZ/2);\bfF_5)$ are trivial in low degrees, because of the exceptional isomorphism between $\Sp_4(\bfZ/2)$ and the symmetric group in $6$ letters (see e.g.\,\cite[p.\,37]{OMeara}).
\end{proof}

\begin{rem}\ \phantom{xyz}
\begin{enumerate}
\item There are at least two other proofs for the divisibility of the signature of classes in $\oH_2(\Sp_{2g}^q(\bfZ);\bfZ)$ by $8$. One can be extracted from the proof of \cite[Lem.\,7.5 i)]{GRWabelian} and another one is given in \cite[Thm\,12.1]{Benson}. The proof in \cite{GRWabelian} shows actually something stronger, namely that the form $\langle-,-\rangle_f$ associated to a class $[f]\in\oH_2(\Sp^q_{2g}(\bfZ);\bfZ)$ is always even. We shall give a different proof of this fact as part of the second part of \cref{lemma:algebraicpontryaginnumbers} below.
 
\item For $g\ge4$, the existence of a class in $\oH_2(\Sp_{2g}^q(\bfZ);\bfZ)$ of signature $8$ was shown as part of the proof of \cite[Thm\,7.7]{GRWabelian}, using that the image of \[\oH_2(\Sp_{2g}(\bfZ,2);\bfZ)\lra\oH_2(\Sp_{2g}(\bfZ);\bfZ)\cong\bfZ\] for $g\ge4$ is known to be $2\cdot\bfZ$ by a result of Putman \cite[Thm\,F]{Putman}. However, this argument breaks for small values of $g$, in which case the image of $\sgn\colon\oH_2(\Sp_{2g}^q(\bfZ);\bfZ)\ra\bfZ$ was not known before, at least to the knowledge of the author.
\end{enumerate}
\end{rem}

By \cref{lemma:algebraicsignatures}, the signatures of classes in $\oH_2(\Sp_{2g}^q(\bfZ);\bfZ)$ are divisible by $8$, so we obtain a morphism 
\[\sgn/8\colon \oH_2(\Sp_{2g}^{q}(\bfZ);\bfZ)\lra\bfZ.\] To lift this morphism to a cohomology class in $\oH^2(\Sp_{2g}^q(\bfZ);\bfZ)$, we consider the morphism \begin{equation}\label{equation:abelianisationthetagroup}a\colon \bfZ/4\lra\Sp_2^q(\bfZ)\subset\Sp_{2g}^q(\bfZ),\end{equation} induced by the matrix $\left(\begin{smallmatrix} 0 & -1 \\ 1 & 0 \end{smallmatrix} \right)\in\Sp_2^q(\bfZ)$. By \cref{lemma:abelianisationSp}, this is an isomorphism on abelianisations for $g\ge3$  and thus induces a splitting of the universal coefficient theorem \begin{equation}\label{equation:splitting1}a^*\oplus h\colon \oH^2(\Sp_{2g}^q(\bfZ);\bfZ)\xlra{\cong} \oH^2(\bfZ/4;\bfZ)\oplus\Hom(\oH_2(\Sp^q_{2g}(\bfZ);\bfZ),\bfZ).\end{equation} This splitting is compatible with the inclusion $\Sp_{2g}^q(\bfZ)\subset\Sp^q_{2g+2}(\bfZ)$, so we can define a lift of the divided signature $\sgn/8$ to a class $\oH^2(\Sp_{2g}^q(\bfZ);\bfZ)$ as follows.

\begin{dfn}\label{definition:signatureclasses}\ \phantom{xyz}
\begin{enumerate}
\item Define the class \[\textstyle{\frac{\sgn}{8}\in\oH^2(\Sp_{2g}^q(\bfZ);\bfZ)}\] for $g\gg0$ via the splitting \eqref{equation:splitting1} by declaring its image in the first summand to be trivial and to be $\sgn/8$ in the second. For small $g$, the class $\frac{\sgn}{8}\in\oH^2(\Sp^q_{2g}(\bfZ);\bfZ)$ is defined as the pullback of $\frac{\sgn}{8}\in\oH^2(\Sp^q_{2g+2h}(\bfZ);\bfZ)$ for $h\gg0$.
\item Define the class \[\textstyle{\frac{\sgn}{8}\in\oH^2(\Gamma^n_{g,1/2};\bfZ)}\quad\text{for $n\neq1,3,7$ odd}\] as the pullback of the same-named class along the map $\Gamma^n_{g,1/2}\ra G_g\cong\Sp_{2g}^q(\bfZ)$ induced by the action on the middle cohomology.
\end{enumerate}
\end{dfn}

\subsection{Framing obstructions}\label{section:algebraicobstructions}
To describe the invariant \[\chi^2\colon\oH_2(\BDiffuo_{\partial/2}(W_{g,1});\bfZ)\lra\bfZ\] explained in \cref{section:obstructionclass} more algebraically, note that a map $f\colon S\rightarrow \oB(\bfZ^{2g}\rtimes\Sp_{2g}(\bfZ))$ from an oriented closed connected based surface $S$ to $\oB(\bfZ^{2g}\rtimes\Sp_{2g}(\bfZ))$ induces a $1$-cocycle 
\[\pi_1(S;*)\lra \bfZ^{2g}\rtimes\Sp_{2g}(\bfZ)\lra  \bfZ^{2g}\] and hence a class $[f]\in\oH^1(S; f^*\cH(g))$. The composition $S\ra\oB(\bfZ^{2g}\rtimes\Sp_{2g}(\bfZ))\ra\BSp_{2g}(\bfZ)$ defines a bilinear form $\langle -,-\rangle_f$ on $\oH^1(S; f^*\cH(g))$ as explained in \cref{section:algebraicsignature}, and hence a number $\langle [f],[f]\rangle_f\in\bfZ$. Varying $f$, this gives a morphism \[\chi^2\colon \oH_2(\bfZ^{2g}\rtimes\Sp_{2g}(\bfZ);\bfZ)\lra\bfZ,\] which is compatible with natural inclusion $\bfZ^{2g}\rtimes\Sp_{2g}(\bfZ)\subset \bfZ^{2g+2}\rtimes\Sp_{2g+2}(\bfZ)$ and takes for $n\equiv3\Mod{4}$ part in a composition
\[\oH_2(\Gamma^n_{g,1/2};\bfZ)\xra{(s_F,p)}\oH_2(\bfZ^{2g}\rtimes\Sp_{2g}(\bfZ);\bfZ)\xlra{\chi^2}\bfZ,\] where the first morphism is induced by acting on a stable framing $F$ of $W_{g,1}$ as in \cref{section:firstextension}. A priori, this requires three choices: a stable framing, a generator $\pi_n\SO\cong\bfZ$, and a symplectic basis $H(g)\cong \bfZ^{2g}$. However, the composition turns out to not be affected by these choices and the following proposition shows that it is related to the invariant of $(W_{g,1},D^{2n-1})$-bundles explained in \cref{example:chiexamples}.
\begin{prop}\label{lemma:pontryaginnumberidentification}For $n\equiv 3\Mod{4}$, the composition
\[\oH_2(\BDiffuo_{\partial/2}(W_{g,1});\bfZ) \longtwoheadrightarrow\oH_2(\Gamma_{g,1/2}^n;\bfZ)\xra{(s_F,p)}\oH_2(\bfZ^{2g}\rtimes\Sp_{2g}(\bfZ);\bfZ)\xra{\chi^2}\bfZ\] 
sends the class of a $(W_{g,1},D^{2n-1})$-bundle $\pi\colon E\ra S$ to $\chi^2(T_\pi E)$.
\end{prop}

\begin{proof}
The relative Serre spectral sequence of \[(W_{g,1},\partial W_{g,1})\lra (E,\partial E)\xlra{\pi} S\] induces canonical isomorphisms 
\[\oH^{2n+2}(E,\partial E;\bfZ)\cong \oH^2(S;\bfZ)\quad\text{and}\quad\oH^{n+1}(E,\partial E;\bfZ)\cong \oH^1(S;f^*\cH(g)),\] 
where $f$ denotes the composition \[S\lra \BDiffuo_{\partial/2}(W_{g,1})\lra\oB\Gamma_{g,1/2}^n\xlra{p}\BSp_{2g}(\bfZ).\] By the compatibility of the Serre spectral sequence with the cup-product and after identifying $\oH^1(S;f^*\cH(g))$ with $\oH^1(\pi_1(S;*);H(g)\otimes S\pi_n\SO(n))$, it suffices to show that the second isomorphism sends $\chi(T_\pi E)\in \oH^{n+1}(E,\partial E;\bfZ)$ up to signs to the class represented by the cocycle
\begin{equation}\label{equ:cocycleobs}
\pi_1(S;*) \lra\Gamma^n_{g,1/2}\xlra{s_F}[W_{g,1},\SO]_*\cong H(g)\otimes \pi_n\SO,\end{equation}
involving the choice of stable framing $F\colon TW_{g,1}\oplus \varepsilon^k\cong \varepsilon^{2n+k}$ as in \cref{section:firstextension}. As a first step, we describe this isomorphism more explicitly: 

Note that $\oH^{n+1}(E,\partial E;\pi_n\SO)\cong \oH^{n+1}(E;\pi_n\SO)$ as $\oH^{*}(\partial E;\pi_n\SO)$ is trivial for $*=n,n+1$. Unwinding the construction of the Serre spectral sequence using a skeletal filtration of $S$, one sees that after fixing an identification $W_{g,1}\cong \pi^{-1}(*)\subset E$, the image of a class $x\in \oH^{n+1}(E,\pi_n\SO)$ under the isomorphism in question is represented by the cocycle $\pi_1(S;*)\ra H(g)\otimes \pi_n\SO$ which maps a loop $\omega\colon ([0,1],\{0,1\})\ra (S,*)$ to the class obtained from a choice of lift $\widetilde{x}\in \oH^{n+1}(E,\pi^{-1}(*);\pi_n\SO)$ by pulling it back along
\[(W_{g,1}\times [0,1],W_{g,1}\times\{0,1\})\lra(\omega^*E,W_{g,1}\times \{0,1\})\lra (E,\pi^{-1}(*)),\] where the second morphism is induced by pulling back the bundle $\pi\colon E\ra S$ along $\omega$ and the first morphism is the unique (up to homotopy) trivialisation $\omega^*E\cong [0,1]\times W_{g,1}$ relative to $W_{g,1}\times\{0\}$ of the pullback bundle over $[0,1]$; here we used the canonical isomorphism $\oH^{n+1}(W_{g,1}\times [0,1],W_{g,1}\times\{0,1\};\pi_n\SO)\cong H(g)\otimes \pi_n\SO$.

Recall from \cref{example:chiexamples} that the class $\chi(T_\pi E)\in \oH^{n+1}(E;\pi_n\SO)$ is the primary obstruction to extending the canonical $\tau_{>n}\BO$-structure on $T_\pi E$ to a $\tau_{>n+1}\BO$-structure. The choice of framing $F$ induces such an extension on the fibre $W_{g,1}\cong\pi^{-1}(*)\subset E$ and thus induces a lift of $\chi(T_\pi E)$ to a relative class in $\oH^{n+1}(E,\pi^{-1}(*);\pi_n\SO)$. This uses that the $\tau_{>n}\BO$ structure on $W_{g,1}$ induced by the framing agrees with the restriction of the $\tau_{>n}\BO$-structure on $T_\pi E$ to $\pi_{-1}(*)\cong W_{g,1}$ by obstruction theory. Using the above description, we see that the image of $\chi(T_\pi E)$ under the isomorphism in question is thus represented by the cocyle $\pi_1(S;*)\ra H(g)\otimes\pi_n\SO$ that sends a loop $\omega$ to the primary obstruction to solving the lifting problem
\begin{equation}\label{equ:liftingproblem}
\begin{tikzcd}
W_{g,1}\times\{0,1\}\dar{\subset}\rar&\pi^{-1}(*)\dar \rar&\tau_{>n+1}\BSO\dar\\
W_{g,1}\times[0,1]\rar\arrow[urr,dashed,crossing over]&E\rar{T_\pi E}&\BSO,
\end{tikzcd}
\end{equation}
in $\oH^n(W_{g,1};\pi_n\SO)\cong H(g)\otimes\pi_n\SO $, since this obstruction agrees with the corresponding obstruction when replacing $\BSO$ by $\tau_{>n}\BSO$. Here the left square in the diagram is given by the trivialisation explained earlier. There is a useful alternative description of this obstruction class: relative to the subspace $W_{g,1}\times\{0\}\subset W_{g,1}\times\{0,1\}$ there is a unique (up to homotopy) lift in \eqref{equ:liftingproblem}, so the obstruction to finding a lift relative to the subspace $W_{g,1}\times \{0,1\}$ can be seen as an element in the group of path-components of the fibre of the principal fibration $\Maps_*(W_{g,1},\tau_{>n+1}\BSO)\ra \Maps_*(W_{g,1},\BSO)$, which is exactly $[W_{g,1},\tau_{\le n}\SO]_*\cong H(g)\otimes \pi_n\SO$.

To see that the cocyle we just described agrees with \eqref{equ:cocycleobs}, note that the function $s_F\colon \Gamma^n_{g,1/2}\ra H(g)\otimes \pi_n\SO\cong[W_{g,1},\tau_{\le n}\SO]_*$ induced by acting on the stable framing $F$ arises as the connecting map $\pi_1(\BDiffuo^{\tau_{>n+1}}_{\partial/2}(W_{g,1});[F])\ra \pi_0\Maps_*(W_{g,1},\tau_{\le n}\SO)$ of the fibration
 \[\Maps_*(W_{g,1},\tau_{\le n}\SO)\lra \BDiffuo^{\tau_{>n+1}}_{\partial/2}(W_{g,1})\lra \BDiffuo_{\partial/2}(W_{g,1}),\] 
 where $\BDiffuo^{\tau_{>n+1}}_{\partial/2}(W_{g,1})$ is the space that classifies $(W_{g,1},D^{2n-1})$-bundles with a $\tau_{>n+1}\BO$-structure on the vertical tangent bundle extending the given $\tau_{>n+1}\BO$-structure on the restriction to the trivial $D^{2n-1}$-subbundle induced by the standard framing of $D^{2n-1}$; here we identified the space of $\tau_{>n+1}\BO$-structures of $W_{g,1}$ relative to $D^{2n-1}$ with the mapping space $\Maps_*(W_{g,1},\tau_{\le n}\SO)$ by using the choice of stable framing $F$, which also induces the basepoint $[F]\in \BDiffuo^{\tau_{>n+1}}_{\partial/2}(W_{g,1})$. This shows that the value of \eqref{equ:cocycleobs} on a loop $[\omega]\in\pi_1(S;*)$ is given by the component in $[W_{g,1},\tau_{\le n}\SO]_*\cong H(g)\otimes\pi_n\SO$ obtained by evaluating a choice of path-lift
 \begin{center}
 \begin{tikzcd}
 \{0\}\arrow[r,]\arrow[d,"\subset"]&*\arrow[r,"F"]\dar&\BDiffuo^{\tau_{>n+1}}_{\partial/2}(W_{g,1})\dar\\
{\left[0,1\right]}\arrow[r,"\omega"]\arrow[urr,dashed, crossing over]&S\arrow[r]&\BDiffuo_{\partial/2}(W_{g,1}).
 \end{tikzcd}
 \end{center}
 at the end point. Such a path-lift precisely classifies a lift as in \eqref{equ:liftingproblem} relative to the subspace $W_{g,1}\times\{0\}\subset W_{g,1}\times\{0,1\}$, so the claim follows from the second description of the obstruction to solving the lifting problem \eqref{equ:liftingproblem} mentioned above.
\end{proof}
 
\begin{lem}\label{lemma:algebraicpontryaginnumbers}
Let $g\ge1$.
\begin{enumerate}
\item The image of the composition induced by the inclusion $\bfZ^{2g}\subset \bfZ^{2g}\rtimes \Sp^q_{2g}(\bfZ)\subset\bfZ^{2g}\rtimes \Sp_{2g}(\bfZ)$
 \[\oH_2(\bfZ^{2g};\bfZ)\lra\oH_2(\bfZ^{2g}\rtimes \Sp^q_{2g}(\bfZ);\bfZ)\lra \oH_2(\bfZ^{2g}\rtimes \Sp_{2g}(\bfZ);\bfZ)\xlra{\chi^2}\bfZ\]  contains $2\cdot\bfZ$ and agrees with $2\cdot\bfZ$ for $g=1$.
\item We have
\[
\begin{gathered}
\im\left(\oH_2(\bfZ^{2g}\rtimes \Sp^{q}_{2g}(\bfZ) ;\bfZ)\xra{(\sgn,\chi^2)}\bfZ\oplus\bfZ\right )=
\begin{cases}\langle(0,2)\rangle &\mbox{if }g=1\\
\langle (8,0),(0,2)\rangle&\mbox{if }g\ge2.
\end{cases}
\\
\im\left(\oH_2(\bfZ^{2g}\rtimes \Sp_{2g}(\bfZ) ;\bfZ)\xra{(\sgn,\chi^2)}\bfZ\oplus\bfZ\right )=
\begin{cases}\langle(0,2)\rangle &\mbox{if }g=1\\
\langle (4,0),(0,1)\rangle&\mbox{if }g\ge2.
\end{cases}
\end{gathered}
\]
\item For $n=3,7$, we have \[\im\left(\oH_2(\Gamma_{g,1/2}^n ;\bfZ)\xlra{(s_F,p)_*}\oH_2(\bfZ^{2g}\rtimes \Sp_{2g}(\bfZ) ;\bfZ)\xra{\chi^2-\sgn}\bfZ\right)=8\cdot\bfZ.\]
\end{enumerate}
\end{lem}

\begin{proof}
By the compatibility of $\chi^2$ with the stabilisation maps, it suffices to show the first part for $g=1$, which follows from checking that the image of a generator in $\oH_2(\bfZ^{2} ;\bfZ)\cong\bfZ$ is mapped to $\pm2$ under $\chi^2$ by chasing through the definition. As the signature morphism pulls back from $\Sp_{2g}(\bfZ)$, the second part follows from the first part and \cref{lemma:algebraicsignatures} by showing that the image of $\chi^2$ is 
for $\Sp_{2g}^q(\bfZ)$ always divisible by $2\cdot\bfZ$ and for $\Sp_{2g}(\bfZ)$ divisible by $2$ if and only if $g=1$. In the case of $\Sp_{2g}^q(\bfZ)$, this can be shown ``geometrically'' as in the proof of \cref{lemma:algebraicsignatures}: choose $n\equiv3\Mod{4}, n\neq3,7$ and consider the composition
\[\oH_2(\BDiffuo_{\partial/2}(W_{g,1});\bfZ) \longtwoheadrightarrow\oH_2(\Gamma^n_{g,1/2};\bfZ)\xlra{(s_F,p)}\oH_2(\bfZ^{2g}\rtimes \Sp^{q}_{2g}(\bfZ) ;\bfZ)\xlra{\chi^2}\bfZ.\] The first morphism is surjective by the second reminder at the beginning of \cref{section:extensionclass} and the second morphism is an isomorphism as a result of \cref{section:firstextension}, so it suffices to show that the composition is divisible by $2$, which in turn follows from a combination of \cref{lemma:pontryaginnumberidentification}, \cref{thm:almostclosedbundles} and the fact that in these dimensions, an $n$-connected almost closed $(2n+2)$-manifold $E'$ satisfies $\chi^2(E')\in 2\cdot\bfZ$ by a combination of \cref{thm:Wall} and \eqref{equation:numbersalmostclosed}. For $\Sp_{2g}(\bfZ)$, we argue as follows: in the case $g=1$, we show that the morphism $\oH_2(\bfZ^2;\bfZ)\ra \oH_2(\bfZ^{2}\rtimes \Sp_{2}(\bfZ);\bfZ)$ is surjective, which will exhibit the claimed divisibility as a consequence of (i). It follows from \cref{lemma:symplecticlowdegreecohomology} that the group $\oH_1(\Sp_{2}(\bfZ);\bfZ^2)$ vanishes and as $\Sp_2(\bfZ)=\SL_2(\bfZ)$, we also have $\oH_2(\Sp_2(\bfZ);\bfZ)=0$\footnote{This can be seen for instance from the Meyer--Vietoris sequence of the well-known decomposition $\SL_2(\bfZ)\cong \bfZ/4\ast_{\bfZ/2}\bfZ/6$, where $\bfZ/4$ is generated by $\left(\begin{smallmatrix}0&1\\-1&0\end{smallmatrix}\right)$, $\bfZ/6$ by $\left(\begin{smallmatrix}1&1\\-1&0\end{smallmatrix}\right)$, and $\bfZ/2$ by $\left(\begin{smallmatrix}-1&0\\0&-1\end{smallmatrix}\right)$.}, so an application of the Serre spectral sequence to $\bfZ^{2}\rtimes \Sp_{2}(\bfZ)$ shows the claimed surjectivity. This leaves us with proving that $\chi^2$ is not divisible by $2$ for $g\ge2$ for which we use that there is class $[f\colon \pi_1S\ra \Sp_{2g}(\bfZ)]\in \oH_2(\Sp_{2g}(\bfZ);\bfZ)$ of signature $4$ by \cref{lemma:algebraicsignatures}, so the form $\langle -,-\rangle_f$ cannot be even and hence there is a $1$-cocycle $g\colon \pi_1S\ra \bfZ^{2g}$ for which $\langle [g],[g]\rangle_f$ is odd, which means that the image $\chi^2([g,f])$ of the class $[(g,f)]\in \oH_2(\bfZ^{2g}\rtimes \Sp_{2g}(\bfZ);\bfZ)$  induced by the morphism $(g,f)\colon \pi_1S\ra \bfZ^2\rtimes \Sp_{2g}(\bfZ)$  is odd. For the last part, note that the argument we gave for the divisibility in the second part for $\Sp_{2g}^q(\bfZ)$ shows for $n=3,7$ that the image of the composition in (iii) is contained in $8\cdot\bfZ$, since $\chi^2(E')-\sgn(E')$ is divisible by $8$ if $n=3,7$. Hence, to finish the proof, it suffices to establish the existence of a class in $\oH_2(\Gamma^n_{g,1/2};\bfZ)$ for $n=3,7$ on which the composition evaluates to $8$. To this end, we consider the square \begin{center}
\begin{tikzcd}
\oH_2(\bfZ^{2g}\otimes S\pi_n\SO(n);\bfZ)\arrow[r]\arrow[d]&\oH_2(\Gamma^n_{g,1/2};\bfZ)\arrow[d,"{(s_F,p)}",swap]\\
\oH_2(\bfZ^{2g}\otimes \pi_n\SO;\bfZ)\arrow[r]&\oH_2((\bfZ^{2g}\otimes \pi_n\SO)\rtimes \Sp_{2g}(\bfZ);\bfZ),
\end{tikzcd}
\end{center}induced by the embedding $(s_F,p)$ of the extension describing $\Gamma_{g,1/2}^n$ into the trivial extension of $\Sp_{2g}(\bfZ)$ by $\bfZ^{2g}\otimes \pi_n\SO$ (see \cref{section:firstextension}). By the first part, there is a class $[f]\in\oH_2(\bfZ^{2g}\otimes \pi_n\SO;\bfZ)$ with $\chi^2([f])=2$ and trivial signature, since the signature morphism pulls back from $\Sp_{2g}(\bfZ)$. As a result of \cref{lemma:stabiliseorthogonalgroups}, the cokernel of the left vertical map in the square is $4$-torsion if $n=3,7$, so $4\cdot[f]$ lifts to $\oH_2(\bfZ^{2g}\otimes S\pi_n\SO(n);\bfZ)$ and provides a class as desired. \end{proof}

Similar to the construction of $\frac{\sgn}{8}\in\oH^2(\Sp_{2g}^q(\bfZ) ;\bfZ )$, we would like to lift the morphism \[\chi^2/2\colon \oH_2(\bfZ^{2g}\rtimes\Sp_{2g}^q(\bfZ) ;\bfZ)\ra \bfZ\] resulting from the second part of \cref{lemma:algebraicpontryaginnumbers} to a class $\frac{\chi^2}{2}\in\oH^2(\bfZ^{2g}\rtimes\Sp_{2g}^q(\bfZ) ;\bfZ)$. To this end, observe that $\Sp_{2g}^q(\bfZ)\subset\bfZ^{2g}\rtimes \Sp_{2g}^q(\bfZ)$ induces an isomorphism on abelianisations for $g\ge2$ since the coinvariants $(\bfZ^{2g})_{\Sp_{2g}^q(\bfZ)}$ vanish in this range by \cref{lemma:coinvariantssymplectic}. The morphism 
\[a\colon \bfZ/4\lra\Sp_2^q(\bfZ)\subset\bfZ^{2g}\rtimes \Sp^q_{2g}(\bfZ)\]
 considered in \cref{section:algebraicsignature} thus induces a splitting \begin{equation}\label{equation:splitting2}a^*\oplus h\colon\oH^2(\bfZ^{2g}\rtimes\Sp_{2g}^q(\bfZ) ;\bfZ)\xra{\cong} \oH^2(\bfZ/4 ;\bfZ)\oplus\Hom(\oH_2(\bfZ^{2g}\rtimes\Sp_{2g}^q(\bfZ) ;\bfZ),\bfZ),\end{equation} for $g\ge3$, analogous to the splitting \eqref{equation:splitting1}. As before, this splitting is compatible with the inclusions $\bfZ^{2g}\rtimes\Sp_{2g}^q(\bfZ)\subset\bfZ^{2g+2}\rtimes\Sp_{2g+2}^q(\bfZ)$, so the following definition is valid.

\begin{dfn}\label{definition:chiclasses}\ \phantom{xyz}
\begin{enumerate}
\item Define the class \[\textstyle{\frac{\chi^2}{2}\in\oH^2(\bfZ^{2g}\rtimes \Sp_{2g}^q(\bfZ);\bfZ)}\] for $g\gg0$ via the splitting \eqref{equation:splitting2} by declaring its image in the first summand to be trivial and to be $\chi^2/2$ in the second. For small $g$, the class $\frac{\chi^2}{2}$ is defined as the pullback of the class for $g\gg0$.
\item Define the class \[\textstyle{\frac{\chi^2}{2}\in\oH^2(\Gamma^n_{g,1};\bfZ)\quad\text{for }n\neq1,3,7\text{ and }n\equiv3\Mod{4}}\] as the pullback of the same-named class along the map \[\Gamma^n_{g,1/2}\xlra{(s_F,p)} \big(H(g)\otimes S\pi_n\SO(n)\big)\rtimes G_g\cong \bfZ^{2g}\rtimes \Sp_{2g}^q(\bfZ).\]
\item Define the class \[\textstyle{\frac{\chi^2-\sgn}{8}\in\oH^2(\Gamma^n_{g,1/2};\bfZ)\cong\Hom(\oH_2(\Gamma^n_{g,1/2} ;\bfZ),\bfZ)\quad\text{for }n=3,7}\] for $g\gg0$ as image of $(\chi^2-\sgn)/8$ ensured by \cref{lemma:algebraicpontryaginnumbers}, and for small $g$ as the pullback of $\frac{\chi^2-\sgn}{8}\in \oH^2(\Gamma_{g+h,1/2};\bfZ)$ for $h\gg0$.
\end{enumerate}
\end{dfn}

The isomorphism $\oH^2(\Gamma^n_{g,1/2};\bfZ)\cong\Hom(\oH_2(\Gamma^n_{g,1/2};\bfZ),\bfZ)$ for $n=3,7$ in the previous definition is assured by \cref{corollary:abelianisationhalfmcg} and \cref{lemma:abelianisationSp}, as $\oH_1(\Gamma^n_{g,1/2};\bfZ)$ vanishes for $g\gg0$. 

\medskip

We finish this subsection with an auxiliary lemma convenient for later purposes.

\begin{lem}\label{lemma:trivialforgone}For $g=1$, the class $\frac{\sgn}{8}\in\oH^2(\Sp^q_{2g}(\bfZ);\bfZ)$ and the pullback of the class $\frac{\chi^2}{2}\in\oH^2(\bfZ^{2g}\rtimes \Sp^q_{2g}(\bfZ);\bfZ)$ to $\oH^2(\Sp^q_{2g}(\bfZ);\bfZ)$ are trivial.
\end{lem}
\begin{proof}
Both classes evaluate trivially on $\oH_2(\Sp_{2g}^{q}(\bfZ);\bfZ)$. For the first class, this is a consequence of \cref{lemma:algebraicsignatures}, for the second this holds by construction. Moreover, both classes pull back trivially to $\oH^2(\bfZ/4;\bfZ)$ along the morphism $a\colon\bfZ/4\ra\Sp^q_{2g}(\bfZ)$ by definition. Although this morphism does not induce an isomorphism on abelianisations for $g=1$, it still induces one on the torsion subgroups of the abelianisations by \cref{lemma:abelianisationSp} and this is sufficient to deduce the claim from the universal coefficient theorem.
\end{proof}

\subsection{The proof of \cref{bigthm:mainextension}}\label{section:proofmainthm}
We are ready to prove our main result \cref{bigthm:mainextension}, which we state equivalently in terms of the central extension \[0\lra\Theta_{2n+1}\lra\Gamma^n_{g}\lra\Gamma^n_{g,1/2}\lra0,\] explained in \cref{section:krecksextensions}. Our description of its extension class in $\oH^2(\Gamma_{g,1/2}^n;\Theta_{2n+1})$ involves the cohomology classes in $\oH^2(\Gamma_{g,1/2}^n;\bfZ)$ of Definitions~\ref{definition:signatureclasses} and~\ref{definition:chiclasses}, and the two homotopy spheres $\Sigma_P$ and $\Sigma_Q$ in the subgroup $\bA_{2n+2}\subset \Theta_{2n+1}$ examined in \cref{section:Wallclassification}. We write $(-)\cdot \Sigma\in \oH^2(\Gamma_{g,1/2}^n;\bfZ)\ra\oH^2(\Gamma_{g,1/2}^n;\Theta_{2n+1})$ for the change of coefficients induced by $\Sigma\in\Theta_{2n+1}$.

\begin{thm}\label{thm:mainextension2}For $n\ge 3$ odd, the extension \[0\lra\Theta_{2n+1}\lra\Gamma_{g,1}^n\lra\Gamma_{g,1/2}^n\lra0\] is classified by the class
\[\begin{cases}
\frac{\sgn}{8}\cdot\Sigma_P&\mbox{if }n\equiv1\Mod{4}\\
\frac{\sgn}{8}\cdot\Sigma_P+\frac{\chi^2}{2}\cdot\Sigma_Q&\mbox{if }n\equiv3\Mod{4},n\neq3,7\\
\frac{\chi^2-\sgn}{8}\cdot\Sigma_Q&\mbox{if }n=3,7.
\end{cases}\]
and its induced differential $d_2\colon \oH_2(\Gamma_{g,1/2}^n)\ra\Theta_{2n+1}$ has image \[\im\left(\oH_2(\Gamma^n_{g,1/2})\xlra{d_2}\Theta_{2n+1}\right)=
\begin{cases}
\bA_{2n+2}&\mbox{if }g\ge2\\ 
\langle \Sigma_Q \rangle&\mbox{if }g=1.
\end{cases}.\] Moreover, the extension splits if and only if $g=1$ and $n\equiv 1\Mod{4}$.
\end{thm}

\begin{proof}As all cohomology classes involved are compatible with the stabilisation map $\Gamma^n_{g,1/2}\ra\Gamma^n_{g+1,1/2}$, it is sufficient to show the first part of the claim for $g\gg0$ (see \cref{section:stabilisation}). We assume $n\neq3,7$ first. Identifying $\Gamma^n_{g,1/2}$ with $(H(g)\otimes \pi_n\SO)\rtimes G_g$ via the isomorphism $(s_F,p)$ of \cref{section:firstextension}, the morphism $a\colon \bfZ/4\ra(H(g)\otimes \pi_n\SO)\rtimes G_g$ of the previous section induces a morphism between the sequences of the universal coefficient theorem
\begin{center}
\begin{tikzcd}[column sep=0.3cm]
0\arrow[r]&\Ext(\oH_1(\Gamma^n_{g,1/2};\bfZ),\Theta_{2n+1})\arrow[d]\arrow[r]&\oH^2(\Gamma^n_{g,1/2};\Theta_{2n+1})\arrow[r]\arrow[d]&\Hom(\oH_2(\Gamma^n_{g,1/2};\bfZ),\Theta_{2n+1})\arrow[d]\arrow[r]&0\\
0\arrow[r]&\Ext(\bfZ/4,\Theta_{2n+1})\arrow[r]&\oH^2(\bfZ/4;\Theta_{2n+1})\arrow[r]&\Hom(\oH_2(\bfZ/4;\bfZ),\Theta_{2n+1})\arrow[r]&0.
\end{tikzcd}
\end{center}
By the exactness of the rows and the vanishing of $\oH_2(\bfZ/4;\bfZ)$, it is sufficient to show that the extension class in consideration agrees with the classes in the statement when mapped to $\Hom(\oH_2(\Gamma^n_{g,1/2};\bfZ),\Theta_{2n+1})$ and $\oH^2(\bfZ/4,\Theta_{2n+1})$. Regarding the images in the Hom-term, it  is enough to identify them after precomposition with the epimorphism 
\[\oH_2(\BDiffuo_{\partial/2}(W_{g,1});\bfZ)\lra\oH_2(\Gamma^n_{g,1/2};\bfZ),\] so from the construction of $\frac{\sgn}{8}$ and $\frac{\chi^2}{2}$ together with Lemmas~\ref{lemma:algebraicsignatures} and~\ref{lemma:algebraicpontryaginnumbers}, we see that it suffices to show that $\oH_2(\BDiffuo_{\partial/2}(W_{g,1});\bfZ)\ra\Theta_{2n+1}$ induced by the extension class maps the class of a bundle $\pi\colon E\ra S$ to $\sgn(E)/8\cdot\Sigma_P$ if $n\equiv 1\Mod{4}$ and to $\sgn(E)/8\cdot\Sigma_P+\chi^2(E)/2\cdot\Sigma_Q$ otherwise, which is a consequence of \cref{thm:almostclosedbundles} combined with \cref{prop:computingboundaries}. By construction, the classes $\frac{\sgn}{8}\cdot\Sigma_P$ and $\frac{\chi^2}{2}\cdot\Sigma_Q$ vanish in $\oH^2(\bfZ/4;\Theta_{2n+1})$, so the claim for $n\neq3,7$ follows from showing that the extension class is trivial in $\oH^2(\bfZ/4;\Theta_{2n+1})$, i.e.\,that the pullback of the extension to $\bfZ/4$ splits, which is in turn equivalent to the existence of a lift 
\begin{center}
\begin{tikzcd}
&\Gamma^n_{g,1}\arrow[d,"{(s_F,p)}"]\\
\bfZ/4\arrow[r]\arrow[ur,dashed,bend left=10]& G_g\subset{(H(g)\otimes \pi_n\SO)\rtimes G_g}.
\end{tikzcd}
\end{center}Using the standard embedding $W_1=S^n\times S^n\subset\bfR^{n+1}\times\bfR^{n+1}$, we consider the diffeomorphism
\[\mapnoname{S^n\times S^n}{S^n\times S^n}{(x_1,\ldots,x_{n+1},y_1,\ldots,y_{n+1})}{(-y_1,\ldots,y_{n+1},x_1,\ldots,x_{n+1}),}\] which is of order $4$, maps to $\left(\begin{smallmatrix} 0 & -1 \\ 1 & 0 \end{smallmatrix} \right)\in\Sp_2^q(\bfZ)$, and has constant differential, so it vanishes in $H(g)\otimes \pi_n\SO$. As the natural map $\Gamma_{1,1}^n\ra\Gamma^n_1$ is an isomorphism by \cref{lemma:fixingadiscdoesnotmatter}, this diffeomorphism induces a lift as required for $g=1$, which in turn provides a lift for all $g\ge1$ via the stabilisation map $\Gamma^n_{1,1}\ra\Gamma^n_{g,1}$. For $n=3,7$, the abelianisation of $\Gamma^n_{g,1/2}$ vanishes for $g\gg0$ due to \cref{corollary:abelianisationhalfmcg}, so it suffices to identify the extension class with the classes in the statement in $\Hom(\oH_2(\Gamma^n_{g,1/2};\bfZ),\Theta_{2n+1})$ which follows as in the case $n\neq 3,7$.

\cref{lemma:algebraicpontryaginnumbers} (iii) implies that the image of $d_2$ for $n=3,7$ is generated by $\Sigma_Q$ for all $g\ge1$. For $n\neq3,7$, the map $\Gamma_{g,1/2}^n\ra (H(g)\otimes\pi_n\SO)\rtimes  G_g$ is an isomorphism (see \cref{section:firstextension}), so \cref{lemma:algebraicpontryaginnumbers} tells us that the image of $d_2$ for $n\equiv3\Mod{4}$ is generated by $\Sigma_P$ and $\Sigma_Q$ if $g\ge2$ and by $\Sigma_Q$ if $g=1$. For $n\equiv1\Mod{4}$, it follows from \cref{lemma:algebraicsignatures} that the image of $d_2$ is generated by $\Sigma_P$ if $g\ge2$ and that it is trivial for $g=1$. In sum, this implies the second part of the claim by \cref{corollary:generatorsbA}, and also that the differential $d_2$ does not vanish for $g\ge2$, so the extension is nontrivial in these cases. For $n\equiv3\Mod{4}$, the homotopy sphere $\Sigma_Q$ is nontrivial by \cref{thm:Schultz}, so $d_2$ does not vanish for $g=1$ either. Finally, in the case $n\equiv1\Mod{4}$, the extension is classified by $\frac{\sgn}{8}\cdot\Sigma_P$, which is trivial for $g=1$ by \cref{lemma:trivialforgone}, so the extension splits and the proof is finished.
\end{proof}

It is time to make good for the missing part of the proof of \cref{thm:splitextension2}.

\begin{proof}[Proof of \cref{thm:splitextension2} for $n=3,7$]
We have $G_g=\Sp_{2g}(\bfZ)$, so the case $g=1$ follows from the fact that $\oH^2(\Sp_{2g}(\bfZ);\bfZ^{2g}\otimes S\pi_n\SO(n))$ vanishes by \cref{lemma:symplecticlowdegreecohomology}. To prove the case $g\ge2$, note that a hypothetical splitting $s\colon \Sp_{2g}(\bfZ)\ra\Gamma^n_{g,1/2}$ of the upper row of the commutative diagram
\begin{center}
\begin{tikzcd}[column sep=0.5cm]
0\arrow[r]&\bfZ^{2g}\otimes S\pi_n\SO(n)\arrow[r]\arrow[d]&\Gamma_{g,1/2}^n\arrow[r]\arrow[d,"{(s_F,p)}"]&\Sp_{2g}(\bfZ)\arrow[r]\arrow[d,equal]&0\\
0\arrow[r]&\bfZ^{2g}\otimes \pi_n\SO\arrow[r]&(\bfZ^{2g}\otimes \pi_n\SO)\rtimes \Sp_{2g}(\bfZ)\arrow[r]&\Sp_{2g}(\bfZ)\arrow[r]&0
\end{tikzcd}
\end{center} induces a splitting $(s_F,p)\circ s$ of the lower row, which agrees with the canonical splitting of the lower row up to conjugation with $\bfZ^{2g}\otimes \pi_n\SO$,  because such splittings up to conjugation are a torsor for $\oH^1(\Sp_{2g}(\bfZ);\bfZ^{2g}\otimes \pi_n\SO)$ which vanishes by \cref{lemma:symplecticlowdegreecohomology}. \cref{lemma:algebraicsignatures} on the other hand ensures that there is a class $[f]\in \oH_2(\Sp_{2g}(\bfZ);\bfZ)$ with signature $4$, so $s_*[f]\in\oH_2(\Gamma^n_{g,1/2};\bfZ)$ satisfies $\sgn(s_*[f])=4$ and $\chi^2(s_*[f])=0$, which contradicts  \cref{lemma:algebraicpontryaginnumbers} (iii).\end{proof}

\section{Kreck's extensions and their abelian quotients}\label{section:applications}
The inclusion $\oT^n_{g,1}\subset \Gamma_{g,1}^n$ of the Torelli group extends to a pullback diagram 
\begin{equation}\label{equation:pullback}
\begin{tikzcd}
0\arrow[r]&\Theta_{2n+1}\arrow[r]\arrow[d, equal]&\oT_{g,1}^n\arrow[r]\arrow[d]&H(g)\otimes S\pi_n\SO(n)\arrow[r]\arrow[d]&0\\
0\arrow[r]&\Theta_{2n+1}\arrow[r]&\Gamma_{g,1}^n\arrow[r]&\Gamma^n_{g,1/2}\arrow[r]&0
\end{tikzcd}
\end{equation}
of extensions whose bottom row we identified in \cref{thm:mainextension2}. We now apply this to obtain information about the top row, and moreover to compute the abelianisations of $\Gamma_{g,1}^n$ and $\oT_{g,1}^n$ in terms of the homotopy sphere $\Sigma_Q\in\Theta_{2n+2}$, the subgroup $\bA_{2n+2}\subset \Theta_{2n+2}$ of \cref{section:highlyconnected}, and the abelianisation of $\Gamma_{g,1/2}^n$ as computed in \cref{corollary:abelianisationhalfmcg}. 

\begin{thm}\label{theorem:bothabelianisations}Let $n\ge3$ be odd and $g\ge1$.
\begin{enumerate}
\item The kernel $K_g$ of the morphism $\Theta_{2n+1}\ra\oH_1(\Gamma^n_{g,1})$ is generated by $\Sigma_Q$ for $g=1$ and agrees with the subgroup $\bA_{2n+2}$ for $g\ge2$. The induced extension
\[0\lra\Theta_{2n+1}/K_g\lra\oH_1(\Gamma_{g,1}^n)\lra\oH_1(\Gamma^n_{g,1/2})\lra0\] splits.
\item The extension
\[0\lra\Theta_{2n+1}\lra\oT^n_{g,1}\lra H(g)\otimes S\pi_n\SO(n)\lra0\] is nontrivial for $n\equiv3\Mod{4}$ and splits $G_g$-equivariantly for $n\equiv1\Mod{4}$. The image of its differential $d_2\colon\oH_2(H(g)\otimes S\pi_n\SO(n);\bfZ)\ra\Theta_{2n+1}$ is generated by $\Sigma_Q$.
\item The commutator subgroup of $\oT^n_{g,1}$ is generated by $\Sigma_Q$ and the resulting extension
\[0\lra\Theta_{2n+1}/\Sigma_Q\lra\oH_1(\oT_{g,1}^n)\lra H(g)\otimes S\pi_n\SO(n)\lra0\] splits $G_g$-equivariantly.
\end{enumerate}
\end{thm}

\begin{proof}
By the naturality of the Serre spectral sequence, the morphism of extension \eqref{equation:pullback} induces a ladder of exact sequences
\begin{center}
\begin{tikzcd}[column sep=0.5cm]
\oH_2(H(g)\otimes S\pi_n\SO(n) ;\bfZ)\arrow[r,"d_2^{\oT}"]\arrow[d]&\Theta_{2n+1}\arrow[r]\arrow[d,equal]&\oH_1(\oT_{g,1}^n)\arrow[r]\arrow[d]& H(g)\otimes S\pi_n\SO(n)\arrow[r]\arrow[d]&0\\
\oH_2(\Gamma_{g,1/2}^n ;\bfZ)\arrow[r,"d_2^\Gamma"]&\Theta_{2n+1}\arrow[r]&\oH_1(\Gamma_{g,1}^n)\arrow[r]& \oH_1(\Gamma_{g,1/2}^n)\arrow[r]&0
\end{tikzcd}
\end{center}
from which we see that the kernel $K_g$ in question agrees with the image of the differential $d_2^\Gamma$, which we described in \cref{thm:mainextension2}. By the universal coefficient theorem, the pushforward of the extension class of the bottom row of \eqref{equation:pullback} along the quotient map $\Theta_{2n+1}\ra\Theta_{2n+1}/\im(d_2^\Gamma)$ is classified by a class in $\Ext(\oH_1(\Gamma^n_{g,1/2}),\Theta_{2n+1}/\im(d_2^\Gamma))$, which also describes the exact sequence in (i). By a combination of \cref{thm:mainextension2} and \cref{corollary:generatorsbA}, this class is trivial for $n=3,7$, and for $n\neq3,7$ as long as $g\ge2$. For $g=1$ and $n\neq3,7$, this class agrees with the image of $\frac{\sgn}{8}\cdot\Sigma_P$ in $\oH^2(\Gamma^n_{g,1};\Theta_{2n+1}/\im(d_2^\Gamma))$ and therefore vanishes as $\frac{\sgn}{8}\in\oH^{2}(G_g;\bfZ)$ is trivial by \cref{lemma:trivialforgone}.

The extension class of the bottom row of \eqref{equation:pullback}, determined in \cref{thm:mainextension2}, pulls back to the extension class of the top one. Since $\frac{\sgn}{8}\in\oH^2(\Gamma^n_{g,1/2};\bfZ)$ is pulled back from $G_g$ by construction, it is trivial in $H(g)\otimes S\pi_n\SO(n)$, so the extension in (ii) is trivial for $n\equiv1\Mod{4}$, classified by $\frac{\chi^2}{2}\cdot\Sigma_Q$ for $n\equiv3\Mod{4}$ if $n\neq3,7$, and by $\frac{\chi^2-\sgn}{8}\cdot\Sigma_Q$ if $n=3,7$. From this, the claimed image of $d_2^T$ follows from \cref{lemma:algebraicpontryaginnumbers} and its proof. For $n\equiv3\Mod{4}$, the homotopy sphere $\Sigma_Q$ is nontrivial, so the extension does not split. This shows the second part of the statement, except for the claim regarding the equivariance, which will follow from the third part, since $\Sigma_Q$ is trivial for $n\equiv1\Mod{4}$ by \cref{thm:Schultz}.

The diagram above shows that the commutator subgroup of $\oT_{g,1}^n$ agrees with the image of $d_2^{\oT}$, which we already showed to be generated by $\Sigma_Q$. To construct a $G_g$-equivariant splitting as claimed, note that the quotient of the lower row of \eqref{equation:pullback} by $\Sigma_Q$ pulls back from $G_g$ by \cref{thm:mainextension2} because $\frac{\sgn}{8}\in\oH^{2}(\Gamma^n_{g,1};\Theta_{2n+1})$ has this property. Consequently, there is a central extension $0\ra \Theta_{2n+1}/\Sigma_Q\ra E\ra G_g\ra0$ fitting into a commutative diagram
\begin{center}
\begin{tikzcd}
0\arrow[r]&\Theta_{2n+1}/\Sigma_Q\arrow[r]\arrow[d,equal]&\oT_{g,1}/\Sigma_Q\arrow[r]\arrow[d]&H(g)\otimes S\pi_n\SO(n)\arrow[r]\arrow[d]&0\\
0\arrow[r]&\Theta_{2n+1}/\Sigma_Q\arrow[r]\arrow[d,equal]&\Gamma^n_{g,1}/\Sigma_Q\arrow[r]\arrow[d]&\Gamma^n_{g,1/2}\arrow[d,"p"]\arrow[r]&0\\
0\arrow[r]&\Theta_{2n+1}/\Sigma_Q\arrow[r]&E\arrow[r]&G_g\arrow[r]&0
\end{tikzcd}
\end{center}
whose middle vertical composition induces a splitting as claimed, using that the right column in the diagram is exact.
\end{proof}
The previous theorem, together with \cref{lemma:fixingadiscdoesnotmatter}, \cref{corollary:abelianisationhalfmcg}, and \cref{thm:Wall} has  \cref{bigthm:kreckextensions} (ii) and \cref{bigthm:abstractsplitting} as a consequence. It also implies the following.

\begin{cor}For $n\ge3$ odd and $g\ge1$, $\oT_{g,1}^n$ is abelian if and only if $n\equiv1\Mod{4}$.
\end{cor}

\subsection{A geometric splitting}\label{section:geometricsplitting}The splittings of the abelianisations of $\Gamma_{g,1}^n$ and $\oT_{g,1}^n$ provided by \cref{theorem:bothabelianisations} are of a rather abstract nature. Aiming towards splitting these sequences more geometrically, we consider the following construction.

A diffeomorphism $\phi\in\Diffuo_\partial(W_{g,1})$ fixes a neighbourhood of the boundary pointwise, so its mapping torus $T_\phi$ comes equipped with a canonical germ of a collar of its boundary $S^1\times \partial W_{g,1}\subset T_\phi$ using which we obtain a closed oriented $(n-1)$-connected $(2n+1)$-manifold $\tilde{T}_\phi$ by gluing in $D^2\times S^{2n-1}$. By obstruction theory and the fact that $W_{g}$ is $n$-parallelisable, the stable normal bundle $\tilde{T}_\phi\ra\BO$ has a unique lift to $\tau_{>n}\BO\ra\BO$ compatible with the lift on $D^2\times S^{2n-1}$ induced by its standard stable framing. This gives rise to a morphism
\[t\colon \Gamma^n_{g,1}\lra\Omega^{\tau_{>n}}_{2n+1},\] which is compatible with the stabilisation map $s\colon \Gamma_{g,1}^n\ra\Gamma_{g,1}^n$ since \[ts(\phi)=[T_\phi\sharp (S^1\times S^n\times S^n)\cup_{S^1\times S^{2n-1}} D^2\times S^{2n-1}]=t(\phi)+[S^1 \times S^n\times S^n]=t(\phi),\] where $S^1\times S^n\times S^n$ carries the $\tau_{>n}\BO$-structure induced by the standard stable framing, which bounds. Using this, it is straight-forward to see that the composition
\[\Theta_{2n+1}=\Gamma_{0,1}^{n}\lra\Gamma_{g,1}^n\xlra{t}\Omega_{2n+1}^{\tau_{>n}}\] of the iterated stabilisation map with $t$ agrees with the canonical epimorphism appearing in Wall's exact sequence \eqref{Wallexactsequence}, so its kernel agrees with the subgroup $\bA_{2n+2}$. Together with the first part of \cref{theorem:bothabelianisations}, we conclude that the dashed arrow in the commutative diagram
\begin{center}
\begin{tikzcd}
0\arrow[r]&\Theta_{2n+1}/K_g\arrow[r]\arrow[dr,bend right,dashed,two heads]&\oH_1(\Gamma_{g,1}^n)\arrow[r,"p_*"]\arrow[d,"t_*"]&\oH_1(\Gamma^n_{g,1/2})\arrow[r]&0\\
&&\Omega_{2n+1}^{\tau_{>n}},&&
\end{tikzcd}
\end{center} is an isomorphism if and only if $K_g=\bA_{2n+2}$ which is the case for $g\ge2$, and for $g=1$ as long as $n=3,7$ since $\Sigma_Q$ generates $\bA_{2n+2}$ for $n=3,7$ by \cref{corollary:generatorsbA}. Consequently, in these cases, the morphism \begin{equation}\label{equation:concretesplitting}t_*\oplus p_*\colon \oH_1(\Gamma_{g,1}^n) \longtwoheadrightarrow\Omega_{2n+1}^{\tau_{>n}}\oplus\oH_1(\Gamma^n_{g,1/2})\end{equation} is an isomorphism, whereas its kernel for $g=1$ coincides with the quotient $\bA_{2n+2}/\Sigma_Q$, which is nontrivial as long as $n\neq3,7$, generated by $\Sigma_P$, and whose order can be interpreted in terms of signatures (see \cref{lemma:quotientbySigmaQ}). Since $p_*$ splits by \cref{theorem:bothabelianisations} and the natural map $\coker(J)_{2n+1}/[\Sigma_Q]\ra\Omega_{2n+1}^{\tau_{>n}}$ is an isomorphism by \cref{corollary:SchultzWall}, the morphism \eqref{equation:concretesplitting} splits for $g=1$ if and only if the natural map\begin{equation}\label{equation:whatIneedtosplit}\Theta_{2n+1}/\Sigma_Q \longtwoheadrightarrow\coker(J)_{2n+1}/[\Sigma_Q]\end{equation} does. Brumfiel \cite[Thm\,1.3]{Brumfiel} has shown that this morphism always splits before taking quotients, so the map \eqref{equation:whatIneedtosplit} (and hence also \eqref{equation:concretesplitting}) in particular splits whenever $[\Sigma_Q]\in\coker(J)_{2n+1}$ is trivial, which is conjecturally always the case (see \cref{conjecture:GRW}) and known in many cases as a result of \cref{thm:Schultz} (see also \cref{rem:BHS}).

The situation for $\oH_1(\oT_{g,1}^n)$ is similar. By \cref{theorem:bothabelianisations}, the morphism $\rho_*$ in the diagram
\begin{center}
\begin{tikzcd}
0\arrow[r]&\Theta_{2n+1}/\Sigma_Q\arrow[r]\arrow[dr,bend right,dashed,two heads]&\oH_1(\oT_{g,1}^n)\arrow[r,"\rho_*"]\arrow[d,"t_*"]&H(g)\otimes S\pi_n\SO(n)\arrow[r]&0\\
&&\Omega_{2n+1}^{\tau_{>n}}&&
\end{tikzcd}
\end{center} splits $G_g$-equivariantly and since the morphism $t_*$ is defined on $\oH_1(\Gamma^n_{g,1})$, its restriction to $\oH_1(\oT^n_{g,1})$ is $G_g$-equivariant when equipping $\Omega_{2n+1}^{\tau_{>n}}$ with the trivial action. By an analogous discussion to the one above, the kernel of the resulting morphism of $G_g$-modules 
 \begin{equation}\label{equation:torellisplitting}t_*\oplus\rho_*\colon \oH_1(\oT_{g,1}^n) \longtwoheadrightarrow\Omega_{2n+1}^{\tau_{>n}}\oplus H(g)\otimes S\pi_n\SO(n)\end{equation} is trivial for $n=3,7$ and given by the quotient $\bA_{2n+2}/\Sigma_Q$ for $n\neq3,7$. Moreover, this morphism splits if and only if it splits $G_g$-equivariantly,\footnote{If $(s_{t},s_{\rho})$ is a non-equivariant splitting of $t_*\oplus \rho_*$, then $(s_t,(\id-s_tt_*)s')$ is an equivariant splitting of $t_*\oplus \rho_*$, where  $s'$ an equivariant splitting of $\rho_*$ ensured by \cref{theorem:bothabelianisations}. This uses that $s_t$ is already equivariant since $\rho_*s_t=0$, so the image of $s_t$ is contained in the image of $\Theta_{2n+1}/\Sigma_Q\ra \oH_1(\oT^n_{g,1})$ which is fixed by the action.} which is precisely the case if the natural map \eqref{equation:whatIneedtosplit} admits a splitting. We summarise this discussion in the following corollary, which has \cref{bigthm:abeliansationmcg} as a consequence when combined with \cref{lemma:quotientbySigmaQ}.
 
\begin{cor}\label{corollary:geometricsplitting}Let $g\ge1$ and $n\ge3$ odd.
\begin{enumerate}
\item The morphism
\[t_*\oplus p_*\colon \oH_1(\Gamma_{g,1}^n) \longtwoheadrightarrow\Omega_{2n+1}^{\tau_{>n}}\oplus\oH_1(\Gamma^n_{g,1/2})\] is an isomorphism for $g\ge2$. For $g=1$, it is an epimorphism and has kernel $\bA_{2n+2}/\Sigma_Q$.
\item The morphism \[t_*\oplus\rho_*\colon \oH_1(\oT_{g,1}^n) \longtwoheadrightarrow\Omega_{2n+1}^{\tau_{>n}}\oplus H(g)\otimes S\pi_n\SO(n)\] is an epimorphism and has kernel $\bA_{2n+2}/\Sigma_Q$.
\item The morphism $t_*\oplus p_*$ splits for $g=1$ if and only if \[\Theta_{2n+1}/\Sigma_Q \longtwoheadrightarrow\coker(J)_{2n+1}/[\Sigma_Q]\] splits, which is the case for $n\equiv1\Mod{4}$. The same holds for $t_*\oplus \rho_*$ for all $g\ge2$.
\end{enumerate}
\end{cor}

\subsection{Abelianising $\Gamma_{g,1}^n$ and $\oT_{g,1}^n$ for $n$ even}\label{section:abelianisationneven}For $n\ge4$ even, the abelianisations of $\Gamma_{g,1}^n$ and $\oT_{g,1}^n$ can be computed without fully determining the extensions \[0\ra\Theta_{2n+1}\ra\Gamma_{g,1}^n\ra\Gamma_{g,1/2}^n\ra0\quad\text{and}\quad 0\ra H(g)\otimes S\pi_n\SO(n)\ra\Gamma_{g,1/2}^n\ra G_g\ra0.\] Indeed, arguing similarly as in the proof of \cref{corollary:abelianisationhalfmcg}, the second extension provides an isomorphism $\oH_1(\Gamma_{g,1/2}^n)\cong\oH_1(G_g)\oplus ( H(g)\otimes S\pi_n\SO(n))_{G_g}$, using that the coinvariants vanish also for $G_g=\oO_{g,g}(\bfZ)$ as long as $g\ge2$ by \cref{lemma:coinvariantssymplectic} and that the extension splits for $g=1$, which is straightforward to check by noting that it is easy to lift elements of
\[G_1=\langle\begin{pmatrix}-1&0\\0&-1\end{pmatrix},\begin{pmatrix}0&1\\1&0\end{pmatrix}\rangle\cong\bfZ/2\oplus\bfZ/2\] to $\Gamma_{1,1}^n\cong\Gamma_n^1\cong\pi_0\Diff(S^n\times S^n)$ (so in particular to $\Gamma_{1,1/2}^n$) using the flip of the factors and a diffeomorphism of $S^n$ of degree $-1$\footnote{There are several mistakes in the literature related to the fact that $G_1=\oO_{1,1}(\bfZ)$ is isomorphic to $(\bfZ/2)^2$ and not $\bfZ/4$: in \cite[p.\,645]{Kreck}, it should be $\tilde{\pi}_0\Diffuo(S^2\times S^2)\cong(\bfZ/2)^2$, in \cite[Thm 1]{Sato} it should be $\tilde{\pi}_0 \Diffuo(S^p\times S^q)/\tilde{\pi}_0S \Diffuo(S^p\times S^q)\cong(\bfZ/2)^2$ for $p=q$ even, and finally in the proof of \cite[Thm\,2.6]{Krylovthesis} it should be $\mathrm{Aut}(\oH_k(S^k\times S^k))\cong (\bfZ/2)^2$ for $k$ even.}. In contrast to the case $n\ge3$, the resulting analogues of the morphisms \eqref{equation:concretesplitting} and \eqref{equation:torellisplitting} for $n\ge4$ even are isomorphisms \emph{for all} $g\ge1$; this is \cref{bigthm:abelianisationeven}. 

\begin{proof}[Proof of \cref{bigthm:abelianisationeven}]Wall's exact sequence \eqref{Wallexactsequence} is also valid for $n\ge4$ even, so implies similarly to the case $n$ odd that both morphisms in question are surjective and that their kernels agree with the quotients of $\bA_{2n+2}$ by the images of the differentials
\[\oH_2(\Gamma_{g,1/2}^n;\bfZ)\xlra{d_2}\Theta_{2n+1}\quad\text{and}\quad  \oH_2(H(g)\otimes S\pi_n\SO(n);\bfZ)\xlra{d_2}\Theta_{2n+1}\] induced by the extensions  \eqref{equation:kreckses2} and \eqref{equation:mainextension2}, so we have to show that these images agree with $\bA_{2n+2}$. By comparing the extensions for different $g$ via the stabilisation map (see \cref{section:stabilisation}) and noting that the second differential factors through the first, we see that it suffices to show this for the second differential in the case $g=1$. Plumbing disc bundles over $S^{n+1}$ defines a pairing of the form
\[S\pi_n\SO(n)\otimes S\pi_n\SO(n)\ra A_{2n+2}^{\tau_{>n}}\] which can be seen to be surjective for $n\ge4$ even, unlike in the case $n\ge3$ odd \cite[p.\,295]{Wall2n+1}. The composition of the pairing with the boundary map $\partial\colon A_{2n+2}^{\tau_{>n}}\ra\Theta_{2n+1}$ with image $\bA_{2n+2}$ is usually called the \emph{Milnor pairing}. Using the surjectivity, it is enough to show that the image of the second differential contains the image of the Milnor pairing. To do so, we rewrite the extension inducing the second differential via the canonical isomorphism $\oT_{1,1}^n\cong \oT_1^n=\ker(\pi_0\Diff(S^n\times S^n)\ra G_1)$ resulting from \cref{lemma:fixingadiscdoesnotmatter} as
\begin{equation}\label{equ:seseven}0\lra \Theta_{2n+1}\xlra{\iota} \oT_1^n \xlra{\rho} S\pi_n\SO(n)^{\oplus 2}\lra 0.\end{equation} and constructsin an explicit cocycle $f\colon S\pi_n\SO(n)^{\oplus 2}\times S\pi_n\SO(n)^{\oplus 2}\ra \Theta_{2n+1}$ classifying this extension as follows: define morphisms \[t_1\colon S\pi_n\SO(n)\ra \oT_1^n\quad\text{and}\quad t_2\colon S\pi_n\SO(n)\ra \oT_1^n\] by assigning to a class $\eta\in S\pi_n\SO(n)$ represented by $\eta\colon (S^n,D^n_+)\ra(\SO(n),\id)\subset (\SO(n+1),\id)$ the diffeomorphism $t_1(\eta)(x,y)=(x,\eta(x)\cdot y)$ and $t_2(\eta)(x,y)=(\eta(y)\cdot x,y)$, where $D^n_+\subset S^n$ is the upper half-disc. Choosing the centre $*\in D^n_+\subset S^n$ as the base point, the diffeomorphisms $t_1(\eta)$ and $t_2(\eta)$ fix both spheres $*\times S^n,S^n\times *\subset S^n\times S^n$, one because $\eta$ is a based map and the other because it factors through the stabilisation map. From the description of  $\rho$ in \cref{section:krecksextensions}, we see that $\rho t_1(\eta)=(\eta,0)$ and $\rho t_2(\eta)=(0,\eta)$, so the function $S\pi_n\SO(n)^{\oplus 2}\ra \oT_1^n$ mapping $(\eta,\xi)$ to $t_1(\eta)\circ t_2(\xi)$ is a set-theoretical section to $\rho$, which implies that 
\[\map{f}{S\pi_n\SO(n)^{\oplus 2}\times S\pi_n\SO(n)^{\oplus 2}}{\Theta_{2n+1}}{\big((\eta_1,\xi_1),(\eta_2,\xi_2)\big)}{t_1(\eta_1) t_2(\xi_1) t_1(\eta_2) t_2(\xi_2) \big(t_1(\eta_1+\eta_2) t_2(\xi_1+\xi_2)\big)^{-1}}\]
defines a $2$-cocycle that represents the extension class of \eqref{equ:seseven} in $\oH^2(S\pi_n\SO(n)^{\oplus 2};\Theta_{2n+1})$ (see e.g.\,\cite[Ch.\,IV.3]{Brown}); here we identified $\Theta_{2n+1}$ with its image in $\oT_1^n$. The differential is the image of this class under the map $\oH^2(S\pi_n\SO(n)^{\oplus 2};\Theta_{2n+1})\ra\Hom(\oH_2(S\pi_n\SO(n)^{\oplus 2});\Theta_{2n+1})$ participating in the universal coefficient theorem (see \cref{section:extensionclass}), so with respect to the canonical isomorphism $\Lambda^2(S\pi_n\SO(n)^{\oplus 2})\cong \oH_2(S\pi_n\SO(n)^{\oplus 2})$ (see e.g.\,\cite[Thm V.6.4 (iii)]{Brown}, the differential takes the form (see e.g.\,\cite[Ex.\,IV.4.8 (c), Ex.\,V.6.5]{Brown})
\[\mapnoname{\Lambda^2(S\pi_n\SO(n)^{\oplus 2})}{\Theta_{2n+1}}{\big((\eta_1,\xi_1)\wedge(\eta_2,\xi_2)\big)}{f\big((\eta_1,\xi_1),(\eta_2,\xi_2)\big)-f\big((\eta_2,\xi_2),(\eta_1,\xi_1)\big).}\] The precomposition of this differential with the map $S\pi_n\SO(n)\otimes S\pi_n\SO(n)\ra \Lambda^2(S\pi_n\SO(n)^{\oplus 2})$ mapping $\eta\otimes\xi$ to $(\eta,0)\wedge(0,\xi)$ thus agrees with the pairing
\[S\pi_n\SO(n)\otimes S\pi_n\SO(n)\ra \Theta_{2n+1}\] that sends $\eta\otimes\xi$ to the commutator $t_1(\eta)t_2(\xi)t_1(\eta)^{-1}t_2(\xi)^{-1}$. By the discussion in \cite[p.\,834]{Lawson}, this coincides with the Milnor pairing, so its image is $\bA_{2n+2}$ by the discussion above, which concludes the claim.
\end{proof}

\subsection{Splitting the homology action}We conclude our study of Kreck's extensions and its abelianisations with the following result, which proves the remaining first part of \cref{bigthm:kreckextensions}. The reader shall be reminded once more of \cref{lemma:fixingadiscdoesnotmatter}, saying that the natural map $\Gamma_{g,1}^n\ra\Gamma^n_g$ is an isomorphism for $n\ge3$.

\begin{thm}\label{theorem:splittingtoG}For $n\ge3$ odd, the extension
\[0\lra\oT_{g,1}^n\lra \Gamma_{g,1}^n\lra G_g\lra 0\]
does not split for $g\ge2$, but admits a splitting for $g=1$ and $n\neq3,7$.
\end{thm}
\begin{proof}
The case $n=3,7$ and $g\ge2$ is clear, since \cref{bigthm:splitextension} shows that under this assumption even the quotient of the extension by the subgroup $\Theta_{2n+1}\subset \oT_{g,1}^n$ does not split. To deal with the other cases, we consider the morphism of extensions
\begin{center}
\begin{tikzcd}
0\arrow[r]&\Theta_{2n+1}\arrow[r]\arrow[d]&\Gamma_{g,1}^n\arrow[r]\arrow[d]&\Gamma^n_{g,1/2}\arrow[r]\arrow[d]&0\\
0\arrow[r]&\oT_{g,1}^n\arrow[r]&\Gamma_{g,1}^n\arrow[r]&G_g\arrow[r]&0.
\end{tikzcd}
\end{center}
By the naturality of the Serre spectral sequence, we have a commutative square of the form
\begin{center}
\begin{tikzcd}
\oH_2(\Gamma^n_{g,1/2};\bfZ)\arrow[d]\arrow[r,"d_2"]&\Theta_{2n+1}\arrow[d]\\
\oH_2(G_g;\bfZ)\arrow[r,"d_2"]&\oH_1(\oT_{g,1}^n;\bfZ)_{G_g}
\end{tikzcd}
\end{center}
whose right vertical map has kernel generated by $\Sigma_Q$ by the computation of $\oH_1(\oT_{g,1}^n;\bfZ)$ as a $G_g$-module in \cref{theorem:bothabelianisations}. Therefore, to finish the proof of the first claim, it suffices to show that the differential $d_2\colon \oH_2(\Gamma^n_{g,1/2};\bfZ)\ra\Theta_{2n+1}/\Sigma_Q$ is nontrivial for $g\ge2$ and $n\neq3,7$, which follows from \cref{thm:mainextension2} together with the fact that $\bA_{2n+2}/\Sigma_Q$ is nontrivial in these cases by \cref{lemma:quotientbySigmaQ} and \cref{remark:minimalsignature}. Turning towards the second claim, we assume $n\neq3,7$ and recall that the isomorphism $(s_F,p)\colon \Gamma^n_{g,1/2}\ra (H(g)\otimes \pi_n\SO)\rtimes G_g$ induces a splitting of the right vertical map $\Gamma^n_{g,1/2}\ra G_g$ in the above diagram (see \cref{section:firstextension}), so the claim follows from showing that the pullback of $\Gamma^n_{g,1}\ra\Gamma_{g,1}^n$ along $G_g\subset (H(g)\otimes \pi_n\SO)\rtimes G_g\cong\Gamma_{g,1}^n$ splits for $g=1$, which is a consequence of \cref{bigthm:mainextension} and \cref{lemma:trivialforgone}.
\end{proof}

\begin{rem}\cref{theorem:splittingtoG} leaves open whether $\Gamma^n_{g,1}\ra G_g$ admits a splitting for $g=1$ in dimensions $n=3,7$. Krylov \cite[Thm\,2.1]{Krylov} and Fried \cite[Sect.\,2]{Fried} showed that this can not be the case for $n=3$ and we expect the same to hold for $n=7$.
\end{rem}

\section{Homotopy equivalences}\label{section:homotopyaut}
Our final result \cref{bigcor:hAut} is concerned with the morphism of extensions
\begin{equation}\label{seshAut}
\begin{tikzcd}
0\arrow[r]& H(g)\otimes S\pi_{n}\SO(n)\arrow[r]\arrow[d]&\Gamma_{g}^n/\Theta_{2n+1}\arrow[r]\arrow[d]&G_g\arrow[r]\arrow[d,equal]& 0\\
0\arrow[r]& H(g)\otimes S\pi_{2n}S^{n}\arrow[r]&\pi_0\hAut(W_g)\arrow[r]&G_g\arrow[r]& 0,
\end{tikzcd}
\end{equation}underlying work of Baues \cite[Thm\,10.3]{Baues}, relating the mapping class group $\Gamma_g^n$ to the group $\pi_0\hAut(W_g)$ of homotopy classes of orientation preserving homotopy equivalences. The left vertical morphism is induced by the restriction of the unstable $J$ homomorphism $J\colon \pi_n\SO(n+1)\ra\pi_{2n+1}S^{n+1}$ to the image of the stabilisation $S\colon \pi_n\SO(n)\ra\pi_n\SO(n+1)$ in the source and to the image of the suspension map $S\colon \pi_{2n}S^{n}\ra \pi_{2n+1}S^{n+1}$ in the target, justified by the fact that the square
\begin{center}
\begin{tikzcd}
\pi_n\SO(n)\arrow[r,"S"]\arrow[d,"J"]&\pi_n\SO(n+1)\arrow[d,"J"]\\
\pi_{2n}S^n\arrow[r,"S"]&\pi_{2n+1}S^{n+1}
\end{tikzcd}
\end{center}
commutes up to sign (see \cite[Cor.\,11.2]{Toda}).

\begin{proof}[Proof of \cref{bigcor:hAut}]By \cref{bigthm:splitextension}, the upper row of \eqref{seshAut} splits for $n\neq1,3,7$ odd, so the first part of (i) is immediate. The second part follows from \cref{bigthm:splitextension} as well if we show that the existence of a splitting of the lower row for $n=3,7$ is equivalent to one of the upper row. To this end, note that for $n=3,7$, we have isomorphisms $S\pi_n\SO(n)\cong\bfZ$ and $S\pi_{2n}S^n=\Tor(\pi_{2n+1}S^{n+1})\cong \bfZ/d_n$ for $d_3=12$ and $d_7=120$ with respect to which the $J$-homomorphism $J\colon S\pi_n\SO(n)\ra S\pi_{n}S^n$ is given by reduction by $d_n$ (see e.g.\,\cite[Ch.\,XIV]{Toda}). By \cref{lemma:symplecticlowdegreecohomology}, the group $\oH^2(G_g;H(g))$ is annihilated by $2$, so in particular by $d_n$. Using this, the first claim follows from the long exact sequence on cohomology induced by the exact sequence
\[0\ra H(g)\otimes S\pi_n\SO(n)\xra{d_n\cdot(-)} H(g)\otimes S\pi_n\SO(n)\xra{J_*} H(g)\otimes S\pi_{2n}S^n\ra0\] of $G_g$-modules, since the extension class of the lower row is obtained from that of the upper one by the change of coefficients $J_*\colon H(g)\otimes S\pi_n\SO(n)\ra H(g)\otimes S\pi_{2n}S^n$. 

To prove the second part, we consider the exact sequence
\[(H(g)\otimes S\pi_{2n}S^{n})_{G_g}\lra\oH_1(\pi_0\hAut(W_g))\lra\oH_1(G_g)\lra 0\] induced by the Serre spectral sequence of the lower row of \eqref{seshAut}. The left morphism in this sequence is split injective as long as the extension splits, which, together with the first part and a consultation of \cref{lemma:coinvariantssymplectic}, exhibits $\oH_1(\pi_0\hAut(W_g))$ to be as asserted.
\end{proof}

\appendix

\section{Low-degree cohomology of symplectic groups}
This appendix contains various results on the low-degree (co)homology of the integral symplectic group $\Sp_{2g}(\bfZ)$ and its theta subgroup $\Sp_{2g}^q(\bfZ)$ (see \cref{section:Wallsform}).
\begin{lem}\label{lemma:abelianisationSp}\ \phantom{xyz}
\begin{enumerate}
\item The abelianisations of $\Sp_{2g}(\bfZ)$ and $\Sp^q_{2g}(\bfZ)$ satisfy
\[\oH_1(\Sp_{2g}(\bfZ))\cong\begin{cases}\bfZ/12&\mbox{if }g=1\\\bfZ/2&\mbox{if }g=2\\0&\mbox{if }g\ge3\end{cases}\quad\text{and}\quad\oH_1(\Sp^q_{2g}(\bfZ))\cong\begin{cases}\bfZ/4\oplus\bfZ&\mbox{if }g=1\\\bfZ/4\oplus \bfZ/2&\mbox{if }g=2\\\bfZ/4&\mbox{if }g\ge3.\end{cases}\]
\item The element $(\begin{smallmatrix} 1 & 1 \\ 0 & 1 \end{smallmatrix})\in\Sp_2(\bfZ)\subset\Sp_{2g}(\bfZ)$ generates $\oH_1(\Sp_{2g}(\bfZ))$ for $g=1,2$.
\item The element $(\begin{smallmatrix} 0 & -1 \\ 1 & 0 \end{smallmatrix}) \in\Sp_2^q(\bfZ)\subset\Sp_{2g}^q(\bfZ)$ generates the $\bfZ/4$-summand in $\oH_1(\Sp_{2g}^q(\bfZ))$ for all $g\ge1$. 
\item The $\bfZ$-summand in $\oH_1(\Sp^q_{2}(\bfZ))$ is generated by $(\begin{smallmatrix} 1 & 2 \\ 0 & 1 \end{smallmatrix})\in\Sp^q_{2}(\bfZ)$ and the $\bfZ/2$-summand in $\oH_1(\Sp^q_{4}(\bfZ))$ is generated by $(\begin{smallmatrix} S & 0 \\ 0 & S \end{smallmatrix})\in \Sp^q_{4}(\bfZ)$, where $S=(\begin{smallmatrix} 0 & 1 \\ 1 & 0 \end{smallmatrix})$.
\item The stabilisation map \[\oH_1(\Sp_{2g}(\bfZ))\lra\oH_1(\Sp_{2g+2}(\bfZ))\] is surjective for all $g\ge1$ and the stabilisation map \[\oH_1(\Sp^q_{2g}(\bfZ))\lra\oH_1(\Sp^q_{2g+2}(\bfZ))\] is surjective for $g\ge2$, but has cokernel $\bfZ/2$ for $g=1$.
\end{enumerate}
\end{lem}
\begin{proof}The fact that the abelianisation of $\Sp_{2}(\bfZ)=\SL_2(\bfZ)$ is generated by $\left(\begin{smallmatrix} 1 & 1 \\ 0 & 1 \end{smallmatrix} \right)$ and of order $12$ is well-known, and so is the isomorphism type of $\oH_1(\Sp_{2g}(\bfZ))$ for $g\ge2$ (see e.g.\,\cite[Lem.\,A.1 (ii)]{BensonCampagnoloRanickiRovi}). The remaining claims regarding $\Sp_{2g}(\bfZ)$ follows from showing that $\oH_1(\Sp_{2}(\bfZ))\ra\oH_1(\Sp_{4}(\bfZ))$ is nontrivial, which can for instance be extracted from the proof of \cite[Thm\,2.1]{GritsenkoHulek}: in their notation $\Gamma_1=\Sp_4(\bfZ)$ and the map $i_\infty\colon \SL_2(\bfZ)\ra\Gamma_1$ identifies with the stabilisation $\Sp_2(\bfZ)\ra\Sp_4(\bfZ)$. The isomorphism type of $\oH_1(\Sp^q_{2g}(\bfZ))$ for $g\ge2$ is determined in \cite[Thm 2]{Endres} and for $g=1$ in \cite[Thm 1]{Weintraub}. The first claim of (ii) follows from the main formula of \cite{JohnsonMillson} for $g\ge3$ and from \cite[Cor.\,2]{Weintraub} for $g=1,2$, which also gives the claimed generator of the $\bfZ$-summand in $\oH_1(\Sp^q_{2}(\bfZ))$. The proof of \cite[Thm 2]{Endres} provides the asserted generator of the $\bfZ/2$-summand in $\oH_1(\Sp^q_{4}(\bfZ))$. The final claim follows from the first four items once we show that the image of $\left(\begin{smallmatrix} 1 & 2 \\ 0 & 1 \end{smallmatrix} \right)$ in $\oH_1(\Sp^q_4(\bfZ))$ vanishes, which is another consequence of the formulas in \cite[Cor.\,2]{Weintraub}. 
\end{proof}

\begin{lem}\label{lemma:coinvariantssymplectic}
The (co)invariants of the standard actions of $\Sp^q_{2g}(\bfZ)$ and $\oO_{g,g}(\bfZ)$ on $\bfZ^{2g}\otimes A$ for an abelian group $A$ satisfy
\[\begin{gathered}(\bfZ^{2g}\otimes A)_{\Sp^q_{2g}(\bfZ)}\cong(\bfZ^{2g}\otimes A)_{\oO_{g,g}(\bfZ)}\cong\begin{cases}A/2&\mbox{if }g=1\\0&\mbox{if }g\ge2\end{cases}\quad\text{and}\\(\bfZ^{2g}\otimes A)^{\Sp^q_{2g}(\bfZ)}\cong(\bfZ^{2g}\otimes A)^{\oO_{g,g}(\bfZ)}\cong\begin{cases}\Hom(\bfZ/2,A)&\mbox{if }g=1\\0&\mbox{if }g\ge2\end{cases}.\end{gathered}\] The same applies to the action of $\Sp_{2g}(\bfZ)$, except that the (co)invariants also vanish for $g=1$.
\end{lem}
\begin{proof} We prove the claim for $\Sp_{2g}^{(q)}(\bfZ)$ first. The self-duality of $\bfZ^{2g}$ induced by the symplectic form implies
\[(\bfZ^{2g}\otimes A)^{\Sp^{(q)}_{2g}(\bfZ)}\cong\Hom(\bfZ^{2g},A)^{\Sp^{(q)}_{2g}(\bfZ)}\cong \Hom((\bfZ^{2g})_{\Sp^{(q)}_{2g}(\bfZ)},A),\] so the computation of the invariants is a consequence of that of the coinvariants. To settle the case $g\ge2$ for both $\Sp^q_{2g}(\bfZ)$ and $\Sp_{2g}(\bfZ)$, it thus suffices to prove that the coinvariants $(\bfZ^{2g}\otimes A)_{\Sp_{2g}^q(\bfZ)}$ with respect to the smaller group vanish. We consider the  matrices \[\left(\begin{smallmatrix}P_\sigma&0\\0&P_\sigma\end{smallmatrix}\right)\in\Sp_{2g}^q(\bfZ)\quad\text{for }\sigma\in\Sigma_g,\quad\left(\begin{smallmatrix}0&-I_g\\I_g&0\end{smallmatrix}\right)\in\Sp_{2g}^q(\bfZ),\quad\text{and}\quad \left(\begin{smallmatrix}1&0&&&&\\1&1&&&&\\&&I_{g-2}&&&\\&&&1&-1&\\&&&0&1&\\&&&&&I_{g-2}\end{smallmatrix}\right)\in\Sp_{2g}^q(\bfZ),\] which can be seen to be contained in the subgroup $\Sp_{2g}^q(\bfZ)\subset \Sp_{2g}(\bfZ)$ using the description in \cref{section:Wallsform}; here $I_g\in\GL_g(\bfZ)$ is the unit matrix and $P_\sigma\in\GL_g(\bfZ)$ the permutation matrix associated to $\sigma\in\Sigma_g$. It suffices to show that $e_i\otimes a$ and $ f_i\otimes a$ are trivial in the coinvariants for $1\le i\le g$ and $a\in A$, where $(e_1,\ldots,e_g,f_1,\ldots,f_g)$ is the standard symplectic basis of $\bfZ^{2g}$. Writing $[-]$ for the residue class of an element in the coinvaraints, acting with the permutation matrices shows $[e_i\otimes a]=[ e_j\otimes a]$ and $[f_i\otimes a]=[f_j\otimes a]$ for all ${1\le i,j\le g}$, so it suffices to prove $[e_1\otimes a]=[ f_1\otimes a]=0$. Using the second matrix, we see that $[ e_1\otimes a]=[f_1\otimes a]$ and $2[ e_1\otimes a]=0$, and finally the third matrix shows $[ e_1\otimes a]=[e_1\otimes a]+[ e_2\otimes a]=2[ e_1\otimes a]=0$, so the coinvariants are trivial. In the case $g=1$, the first part of the proof is still valid and shows that $(\bfZ^{2}\otimes A)_{\Sp_{2}^q(\bfZ)}$ and $(\bfZ^{2}\otimes A)_{\Sp_{2}^q(\bfZ)}$ are quotients of $A/2$ generated by $[ e_1\otimes a]=[ f_1\otimes a]$ for $a\in A$. For $\Sp_{2}(\bfZ)$, we may use $(\begin{smallmatrix}1&1\\0&1\end{smallmatrix})\in\Sp_{2}(\bfZ)$ to conclude $[ f_1\otimes a]=[ e_1\otimes a]+[f_1\otimes a]=2[e_1\otimes a]=0$, so the coinvariants vanish. For $\Sp^q_{2}(\bfZ)$, one uses that $\Sp^q_{2}(\bfZ)$ is generated by $(\begin{smallmatrix}1&2\\0&1\end{smallmatrix})$ and $(\begin{smallmatrix}0&1\\-1&0\end{smallmatrix})$ (see e.g.\,\cite[p.\,385]{Weintraub}) to see that the surjection $\bfZ^{2}\otimes A\ra A/2$ induced by adding coordinates is invariant and surjective, and thus induces an isomorphism \[(\bfZ^{2g}\otimes A)_{\Sp_{2g}^q(\bfZ)}\cong A/2,\] as claimed. The proof for $\oO_{g,g}(\bfZ)$ is almost identical to that for $\Sp_{2g}^q(\bfZ)$, except that one has to replace the second displayed matrix by $(\begin{smallmatrix}0&I\\I&0\end{smallmatrix})$, also act with $-I_{2g}$, and use that $\oO_{1,1}(\bfZ)$ is generated by $(\begin{smallmatrix}-1&0\\-1&0\end{smallmatrix})$ and $(\begin{smallmatrix}0&1\\1&0\end{smallmatrix})$.\end{proof}

\begin{lem}\label{lemma:symplecticlowdegreecohomology}The cohomology groups $\oH^*(\Sp_{2g}(\bfZ),\bfZ^{2g})$ and $\oH^*(\Sp^{q}_{2g}(\bfZ),\bfZ^{2g})$ with coefficients in the standard module are annihilated by $2$ and in low degrees given by
\[ \begin{array}{lll}
\oH^0(\Sp_{2g}(\bfZ);\bfZ^{2g})=0  &  \oH^0(\Sp^q_{2g}(\bfZ);\bfZ^{2g})=0  \\
\oH^1(\Sp_{2g}(\bfZ);\bfZ^{2g})\cong0  &  \oH^1(\Sp^{q}_{2g}(\bfZ);\bfZ^{2g})\cong\begin{cases}\bfZ/2&g=1\\0&g\ge2\end{cases} \\
\oH^2(\Sp_{2g}(\bfZ);\bfZ^{2g})\cong0\ \ \text{for }g=1&  \oH^2(\Sp^{q}_{2g}(\bfZ);\bfZ^{2g})\cong\bfZ/2\oplus\bfZ/2\ \ \text{for }g=1.
\end{array}
\]
\end{lem}
\begin{proof}The negative of the identity matrix $-I\in\Sp^{(q)}_{2g}(\bfZ)$ is central and acts by $-1$ on $\bfZ^{2g}$, so the first claim follows from the ``centre kills trick" which is worth recalling: multiplication $g\cdot(-)\colon M\ra M$ by an element $g\in G$ of a discrete group $G$ on a $G$-module $M$ is equivariant with respect to the conjugation $c_g\colon G\ra G$ by $g$ and induces the identity on $\oH^*(G;M)$ by the usual argument. If $g\in G$ is central, conjugation by $g$ is trivial, so $\oH^*(G;M)$ is annihilated by the action of $(1-g)$ on $M$.

Since $\oH^*(\Sp^{(q)}_{2g}(\bfZ);\bfZ^{2g})$ is $2$-torsion, the self-duality of $\bfZ^{2g}$ combined with the universal coefficient theorem for nontrivial coefficients (see e.g.\,\cite[p.\,283]{Spanier}) implies \[\Ext^1_\bfZ(\oH_{i-1}(\Sp^{(q)}_{2g}(\bfZ);\bfZ^{2g}),\bfZ)\cong\oH^i(\Sp^{(q)}_{2g}(\bfZ);\bfZ^{2g}),\] so the computation of $\oH^*(\Sp^{(q)}_{2g}(\bfZ);\bfZ^{2g})$ for $*\le1$ is a consequence of \cref{lemma:coinvariantssymplectic}. 

From the Lyndon--Hochschild--Serre spectral sequence for the extension \[0\lra\{\pm1\}\lra\Sp_2^{(q)}(\bfZ)\lra\PSp_2^{(q)}(\bfZ)\lra0\] with coefficients in $\bfZ^2$, we see that  
\[\oH^{2}(\Sp^{(q)}_{2}(\bfZ);\bfZ^{2})\cong{\oH^1(\PSp^{(q)}_2(\bfZ);\bfF_2^2)}\] and the right hand side can easily be computed to be as claimed by applying the Mayer--Vietoris sequence to the presentation
\[\PSp_2(\bfZ)=\langle S,T\mid S^2=1,T^3=1\rangle\quad\text{and}\quad\PSp^q_2(\bfZ)=\langle S,R\mid S^2=1\rangle\] for \[S=\begin{pmatrix} 0 & -1 \\ 1 & 0 \end{pmatrix},\quad T=\begin{pmatrix} 0 & -1 \\ 1 & 1 \end{pmatrix},\quad\text{and}\quad R=\begin{pmatrix} 1 & 2 \\ 0 & 1 \end{pmatrix},\] which is well-known for $\PSp_2(\bfZ)\cong \PSL_2(\bfZ)$ and appears e.g.\,in \cite[p.\,385]{Weintraub} for $\PSp^q_2(\bfZ)$.
\end{proof}

\newcommand{\etalchar}[1]{$^{#1}$}
\providecommand{\bysame}{\leavevmode\hbox to3em{\hrulefill}\thinspace}

\providecommand{\href}[2]{#2}

\vspace{0.2cm}

\end{document}